\documentclass{article}
\usepackage[utf8]{inputenc}
\usepackage{changes}
\definechangesauthor[name=Daniele, color=red]{DAN}
\definechangesauthor[name=Lins, color=olive]{LINS}
\usepackage[pagewise]{lineno}
\usepackage{cite}
\usepackage[top=2.5cm,bottom=2.5cm,left=2.5cm,right=3.5cm]{geometry}
\usepackage[english]{babel}
\usepackage{ulem,soul}
\usepackage{tikz}
\usetikzlibrary{calc}
\usetikzlibrary{intersections}

\title{Minimal codewords arising from the incidence of points and hyperplanes in projective spaces}

\author{Daniele Bartoli\thanks{Dipartimento di Matematica e Informatica, Universit\`a degli studi di Perugia,  Perugia, Italy. daniele.bartoli@unipg.it} \and Lins Denaux\thanks{Department of Mathematics: Analysis, Logic and Discrete Mathematics, Ghent University,  Ghent, Belgium. lins.denaux@ugent.be}}
\date{}

\usepackage{amsmath, amssymb, amsthm}
\usepackage{xcolor, comment}

\newtheorem{thm}{Theorem}[section]
\newtheorem*{thm*}{Theorem}
\newtheorem{lm}[thm]{Lemma}
\newtheorem{crl}[thm]{Corollary}
\newtheorem{prop}[thm]{Proposition}

\theoremstyle{definition}
\newtheorem{df}[thm]{Definition}
\newtheorem{constr}[thm]{Construction}
\newtheorem{open}[thm]{Open Problem}

\newcommand{\NN}{\mathbb{N}}
\newcommand{\ZZ}{\mathbb{Z}}
\newcommand{\wt}{\textnormal{wt}}
\newcommand{\supp}{\textnormal{supp}}
\newcommand{\zero}{\mathbf 0}
\newcommand{\vspan}[1]{\left \langle #1 \right \rangle}
\newcommand{\pg}{\textnormal{PG}}
\newcommand{\mc}{\mathcal{C}}
\newcommand{\restr}[2]{{#1}_{|#2}}

\renewcommand{\geq}{\geqslant}
\renewcommand{\leq}{\leqslant}

\begin{document}

\maketitle
\begin{abstract}
Over the past few years, the codes $\mc_{n-1}(n,q)$ arising from the incidence of points and hyperplanes in the projective space $\pg(n,q)$ attracted a lot of attention. In particular, small weight codewords of $\mc_{n-1}(n,q)$ are a topic of investigation. The main result of this work states that, if $q$ is large enough and not prime, a codeword having weight smaller than roughly $\frac{1}{2^{n-2}}q^{n-1}\sqrt{q}$ can be written as a linear combination of a few hyperplanes. Consequently, we use this result to provide a graph-theoretical sufficient condition for these codewords of small weight to be minimal.
\end{abstract}

	

\section{Introduction}

Let $n\in\NN\setminus\{0,1\}$ and $q:=p^h$, with $p$ a prime and $h$ a positive integer. Denote by $\mathbb{F}_q$ the Galois field of order $q$ and by $\pg(n,q)$ the Desarguesian projective space of dimension $n$ over $\mathbb{F}_q$.
Define
\[
    \theta_{m,q}:=\begin{cases}\frac{q^{m+1}-1}{q-1}\quad&\textnormal{if }m\in\NN\textnormal{,}\\0\quad&\textnormal{if }m\in\ZZ\setminus\NN\textnormal{,}\end{cases}
\]
which equals the number of points in $\pg(m,q)$.
Interesting classes of linear error correcting codes can be constructed in this geometric setting; see for instance \cite{Rudolph}.

Let $j,k\in\NN$ be such that $0\leq j<k<n$, and denote by $G_j(n,q)$, respectively $G_k(n,q)$, the set of all $j$-dimensional, respectively $k$-dimensional, subspaces of $\pg(n,q)$.
For each $\kappa\in G_k(n,q)$, we can define a function $f_\kappa$ as follows.
\[
    f_\kappa:G_j(n,q)\rightarrow\mathbb{F}_p:\lambda\mapsto f_\kappa(\lambda):=\begin{cases}1&\textnormal{if }\lambda\subseteq\kappa\textnormal{,}\\0&\textnormal{otherwise.}\end{cases}
\]
Define the $p$-ary linear code $\mc_{j,k}(n,q)$ as the $p$-ary vector subspace of $\mathbb{F}_p^{G_j(n,q)}$ generated by the set $\{f_\kappa:\kappa\in G_k(n,q)\}$.
We will often denote $\mc_{0,k}(n,q)$ by $\mc_k(n,q)$.
These codes belong to the more general class of generalised Reed–Muller codes; see \cite{assmus_key_1992,bagchi,FackFancsaliStorme,MR0465510,MR2766085}.

An element $f_\kappa\in\{f_\kappa:\kappa\in G_k(n,q)\}$ will often be identified by the corresponding $k$-subspace $\kappa$.
As a consequence, for a codeword $c:=\alpha_1f_{\kappa_1}+\alpha_2f_{\kappa_2}+\dots+\alpha_sf_{\kappa_s}$, $\alpha_i\in\mathbb{F}_p$, we will informally describe $c$ as being `a linear combination of the subspaces $\kappa_1,\dots,\kappa_s$'.
\textbf{By convention, if $\boldsymbol{c}$ can be written as a linear combination of the subspaces $\boldsymbol{\kappa_1,\dots,\kappa_s}$, we assume that each of these subspaces appears non-trivially, i.e.\ the corresponding coefficients $\boldsymbol{\alpha_1,\dots,\alpha_s}$ are non-zero.}

For any $c\in\mc_{j,k}(n,q)$, define the \emph{support} of $c$ as $\supp(c):=\{\lambda\in G_j(n,q):c(\lambda)\neq0\}$ and the \emph{weight} of $c$ as $\wt(c):=|\supp(c)|$.
The \emph{minimum weight} of $\mc_{j,k}(n,q)$ is defined as $d\left(\mc_{j,k}(n,q)\right):=\min\{\wt(c):\zero\neq c\in\mc_{j,k}(n,q)\}$. The minimum weight of the codes $\mc_{j,k}(n,q)$ is well known.

\begin{thm}[\!\!{\cite[Theorem 1]{bagchi}}]\label{Th:1.1}
    The minimum weight of $\mc_{j,k}(n,q)$ equals the number of $j$-spaces in a fixed $k$-space. The minimum weight codewords are the scalar multiples of $k$-spaces.
\end{thm}

Recently, small weight codewords of $\mc_{j,k}(n,q)$ were studied and characterised in \cite{AdriaensenDenaux}, in which the authors also investigate the minimal weight problem of the dual code $\mc_{j,k}(n,q)^\perp$.

For the particular code $\mc_{1}(2,q)$, small weight codewords have been characterised, see \cite{Chouinard,FackFancsaliStorme,LAVRAUW2009996,szonyi}; we summarise the most recent results in the following two theorems.

\begin{thm}[\!\!{\cite[Theorem 4.8, Corollary 4.10]{szonyi}}]
    Let $c$ be a codeword of $\mc_{1}(2,p)$, $p >17$ a prime. If $\wt(c) \leq \max\{3p +1,4p-22\}$, then $c$ is either the linear combination of three lines or \cite[Example 4.7]{szonyi}.
\end{thm}

\begin{thm}[\!\!{\cite[Theorem 4.3]{szonyi}}]\label{Res_SzonyiZsuzsa}
    Let $q=p^h$, $h\geq2$, with $q>27$.
    Then any $c\in\mc_1(2,q)$ with
	\begin{itemize}
		\item $\wt(c)<(\lfloor\sqrt{q}\rfloor+1)(q+1-\lfloor\sqrt{q}\rfloor)$, when $h>2$, or
		\item $\wt(c)<\frac{(p-1)(p-4)(p^2+1)}{2p-1}$, when $h=2$,
	\end{itemize}
	is a linear combination of exactly $\left\lceil\frac{\wt(c)}{q+1}\right\rceil$ different lines.
\end{thm}

Recently, results about $\mc_{1}(2,q)$ were extended to $\mc_{n-1}(n,q)$.

\begin{thm}
    \begin{enumerate}
        \item \cite[Theorem 1.4]{POLVERINO20181} There are no codewords with weight in the open interval $]\theta_{n-1,q},2q^{n-1}[$, and the codewords of weight $2q^{n-1}$ in $\mc_{n-1}(n,q)$ are the scalar multiples of the differences of two distinct hyperplanes of $\pg(n,q)$.
        \item \cite[Theorem 3.1.6]{AdriaensenDenauxStorme} If $q$ is large enough, the codewords in $\mc_{n-1}(n,q)$ of weight at most $4q^{n-1}-\mathcal{O}\left(q^{n-2}\sqrt{q}\right)$ can be written as linear combinations of hyperplanes through a common $(n-3)$-space.
    \end{enumerate}
\end{thm}

The bound on the weight of codewords in the latter result seemed hard to improve if $q$ is prime, due to the existence of a peculiar small weight codeword in $\mc_1(2,q)$, $q$ prime \cite[Example 4.7]{szonyi}.
This codeword, however, ceases to exist when $q$ is assumed to be non-prime.
Hence, in this paper, we will focus on the codes $\mc_{n-1}(n,q)$ for $q$ not prime; our main result is the following extension of Theorem \ref{Res_SzonyiZsuzsa} to the case $n\geq3$ (see Section \ref{Sect_MinWeight}).

\begin{thm}\label{Thm_MainHyperplane}
    Let $n\geq3$ and let $q=p^h$, $h\geq2$, with
    \begin{equation}\label{eq:q}
        q\geq\begin{cases}\max\left\{32,2^{2n-4}\right\}\quad&\textnormal{if }h>2\textnormal{,}\\2^{2n}\quad&\textnormal{if }h=2\textnormal{.}\end{cases}
    \end{equation}
    Then any $c\in\mc_{n-1}(n,q)$ with
    \begin{itemize}
        \item $\wt(c)\leq \left(\left\lfloor\frac{1}{2^{n-2}}\sqrt{q}\right\rfloor-1\right)\theta_{n-1,q}$, when $h>2$, or
        \item $\wt(c)\leq \left(\left\lfloor\frac{p}{2^n}\right\rfloor-1\right)\theta_{n-1,q}$, when $h=2$,
    \end{itemize}
    is a linear combination of exactly $\left\lceil\frac{\wt(c)}{\theta_{n-1,q}}\right\rceil$ different hyperplanes.
\end{thm}

Note that if $q\in\{16,27\}$ and $n=3$, then Theorem \ref{Thm_MainHyperplane} follows from the $(j,k)=(0,2)$ case of Theorem \ref{Th:1.1}; all other values of $q$ and $n\geq3$ not satisfying \eqref{eq:q} provide \emph{non-positive} upper bounds on the weights $\wt(c)$ and hence make Theorem \ref{Thm_MainHyperplane} trivially true.
This means that the explicit assumptions on $q$ isn't necessary for the theorem to stay true; we however keep the assumptions \eqref{eq:q} to emphasize that we may assume $q$ to be big.

\bigskip
In Section \ref{Sec:minimal}, we manage to formulate a graph-theoretical sufficient condition for these codewords of small weight to be \emph{minimal} (see Definition \ref{Def_Minimal}). Minimal codewords can be used to describe access structures in linear code-based secret sharing schemes (SSS) (see \cite{Massey1993,Massey1995}), which is a method to distribute shares of a secret to each of the participants $\mathcal{P}$ in such a way that only the authorised subsets of $\mathcal{P}$ (access structure $\Gamma$) could reconstruct the secret; see \cite{Shamir,Blakley}. 

In \cite{Massey1993,Massey1995}, Massey proposed the use of linear codes $\mathcal{C}$ for realising a perfect  and ideal SSS in which the access structure of the secret-sharing is specified by the supports of minimal codewords in $\mathcal{C}^{\bot}$ having $1$ as the first component.

Due to the hardness of determining the set of minimal codewords of a linear code \cite{BMeT1978,BN1990}, research is mainly focused on analysing codes for which every codeword is minimal; see for instance \cite{CCP2014,YuanDing,HengDingZhou,DingHengZhou,CohenMesnagerPatey,BoniniBorello,BartoliBonini,AshikhminBarg}.

\section{Small weight codewords of $\boldsymbol{\mc_{n-1}(n,q)}$, $\boldsymbol{q}$ not prime}\label{Sect_MinWeight}

This section is devoted to prove our main result (Theorem \ref{Thm_MainHyperplane}).
Therefore, assumptions \eqref{eq:q} on $q$  will be used throughout this part of the paper.

As the proof will be done by induction, the following lemma is a relatively trivial but crucial result, and will often be used (without mention) throughout the proofs presented in this section. For a codeword $c\in\mc_{n-1}(n,q)$ and an $i$-subspace $\iota$ of $\pg(n,q)$, we define the \emph{restricted codeword} $\restr{c}{\iota}$ as the codeword $c$ restricted to the points of $\iota$.

\begin{lm}[\!\!{\cite[Remark 3.1]{POLVERINO20181}}]\label{Lm_SubspaceIsCodeWord}
	Let $c\in \mc_{n-1}(n,q)$ be a codeword and $\iota$ an $i$-space of $\pg(n,q)$.
	Then $\restr{c}{\iota}$ is a codeword of $\mc_{i-1}(i,q)$.
\end{lm}

To simplify notation, define the following values for $i\in\{0,1,2,\dots,n\}$ and $q=p^h$ a prime power, $h\geq2$.
\begin{itemize}
    \item $\Delta_{i,q}:=\begin{cases}\left\lfloor\frac{1}{2^{i-2}}\sqrt{q}\right\rfloor\quad&\textnormal{if }h>2\textnormal{,}\\\left\lfloor\frac{p}{2^i}\right\rfloor\quad&\textnormal{if }h=2\textnormal{.}\end{cases}\qquad$(note that $\Delta_{i,q}\geq1$)
    \item $W(i,q):=\left(\Delta_{i,q}-1\right)\theta_{i-1,q}\geq0$.
    \item $U(n,i,q):=q^i-(\Delta_{n,q}-2)\lfloor q^{i-1}\rfloor-(i-2)\big((q-1)\Delta_{n,q}+1\big)\lfloor q^{i-3}\rfloor+\theta_{i-3,q}$.
\end{itemize}

\begin{df}
    Let $c\in\mc_{n-1}(n,q)$ and let $\iota$ be an $i$-subspace of $\pg(n,q)$ ($i\in\{0,1,2,\dots,n\}$).
    \begin{itemize}
        \item If $\wt(\restr{c}{\iota})\leq W(i,q)$, we will call $\iota$ a \emph{thin (sub)space} (with respect to $c$).
        \item If $\wt(\restr{c}{\iota})\geq U(n,i,q)$, we will call $\iota$ a \emph{thick (sub)space} (with respect to $c$).
    \end{itemize}
    If $i=0$ and $\iota$ is thin w.r.t.\ $c$ (i.e.\ $\iota\notin\supp(c)$), we will call $\iota$ a \emph{hole} of $c$.
\end{df}

Keeping these assumptions and definition in mind, together with the ones depicted in Theorem \ref{Thm_MainHyperplane}, the latter theorem reads as follows.

\begin{thm*}
    Any $c\in\mc_{n-1}(n,q)$ with $\wt(c)\leq W(n,q)$ is a linear combination of exactly $\left\lceil\frac{\wt(c)}{\theta_{n-1,q}}\right\rceil$ different hyperplanes.
\end{thm*}

During the proof of the main theorem, we will make use of induction on the dimension $n$.
The base case, stated below in Lemma \ref{Lm_PlaneCase}, was already proved in \cite{szonyi}.

\begin{lm}\label{Lm_PlaneCase}
    Any $c\in\mc_1(2,q)$ with $\wt(c)\leq W(2,q)$ is a linear combination of exactly $\left\lceil\frac{\wt(c)}{q+1}\right\rceil$ different lines.
\end{lm}
\begin{proof}
    Note that $2(q+1)=q+q+2\geq\lfloor\sqrt{q}\rfloor^2+\lfloor\sqrt{q}\rfloor+2>\lfloor\sqrt{q}\rfloor(\lfloor\sqrt{q}\rfloor+1)$.
    Hence, we obtain
    \begin{align*}
        \wt(c)&\leq(\lfloor\sqrt{q}\rfloor-1)(q+1)\\
        &=(\lfloor\sqrt{q}\rfloor+1)(q+1)-2(q+1)\\
        &<(\lfloor\sqrt{q}\rfloor+1)(q+1)-\lfloor\sqrt{q}\rfloor(\lfloor\sqrt{q}\rfloor+1)\\
        &=(\lfloor\sqrt{q}\rfloor+1)(q+1-\lfloor\sqrt{q}\rfloor)
    \end{align*}
    in case $h>2$.
    Note that $\frac{p}{4}-1<\frac{(p-1)(p-4)}{2p-1}$ if $p\geq5$, hence we get
    \begin{align*}
        \wt(c)&\leq\left(\frac{p}{4}-1\right)(p^2+1)\\
        &<\frac{(p-1)(p-4)(p^2+1)}{2p-1}
    \end{align*}
    in case $h=2$.
    The claim follows from Theorem \ref{Res_SzonyiZsuzsa}.
\end{proof}

\begin{prop}\label{Prop_ThickSpaceLowerBound}
    For any $i\in\{1,2,\dots,n\}$,
    \[
        \theta_{i,q}-\Delta_{n,q}q^{i-1}+1\leq U(n,i,q)\textnormal{,}
    \]
    with equality if $i=1$.
\end{prop}
\begin{proof}
    The equality for $i=1$ can be easily checked, hence we may assume that $i\geq2$.
    Moreover, as $\Delta_{n,q}\geq1$, we get
    \[
        (i-2)\big((q-1)\Delta_{n,q}+1\big)\lfloor q^{i-3}\rfloor=(i-2)\big(q\Delta_{n,q}-\Delta_{n,q}+1\big)\lfloor q^{i-3}\rfloor\leq(i-2)\Delta_{n,q}q^{i-2}\textnormal{.}
    \]
    
    As $i\leq n$ and $n\geq3$, we obtain
    \[
        (i-2)\Delta_{n,q}q^{i-2}\leq\frac{n-2}{2^{n-2}}q^{i-2}\sqrt{q}\leq\frac{1}{2}q^{i-2}\sqrt{q}\textnormal{.}
    \]
    Moreover, since $i\geq2$ and $q\geq3$, the inequality $1\leq q^{i-2}\left(q-\frac{1}{2}\sqrt{q}-1\right)$ holds, which yields $\frac{1}{2}q^{i-2}\sqrt{q}\leq q^{i-1}-q^{i-2}-1$.
    In conclusion, we have deduced that
    \[
        (i-2)\big((q-1)\Delta_{n,q}+1\big)\lfloor q^{i-3}\rfloor\leq q^{i-1}-q^{i-2}-1\textnormal{,}
    \]
    which suffices to prove the statement.
\end{proof}

For any point set $\mathcal{P}$ of $\pg(n,q)$ and $m\in\{0,1,\dots,|\mathcal{P}|\}$, an $m$\emph{-secant to} $\mathcal{P}$ is defined to be a line meeting $\mathcal{P}$ in precisely $m$ points.

\begin{lm}\label{Lm_SmallLargeSecants}
    Let $c\in\mc_{n-1}(n,q)$ with $\wt(c)\leq W(n,q)$.
    Then there are no $m$-secants to $\supp(c)$ if
    \[
        \Delta_{n,q}+1\leq m\leq q-\Delta_{n,q}+1\textnormal{.}
    \]
    In particular, every line is either thin or thick with respect to $c$.
\end{lm}
\begin{proof}
    Suppose, to the contrary, that there exists such an $m$-secant $l$ to $\supp(c)$.
    
    First, we prove that any plane through $l$ has to contain at least $\Delta_{n,q}(q+1)+1$ points of $\supp(c)$.
    After all, suppose to the contrary that $\pi$ is a plane through $l$ for which $\wt(\restr{c}{\pi})\leq\Delta_{n,q}(q+1)$.
    As $n\geq3$, we have
    \[
        \wt(\restr{c}{\pi})\leq\Delta_{n,q}(q+1)\leq\frac{1}{2}\sqrt{q}(q+1)\leq q\sqrt{q}-q\textnormal{,}
    \]
    where this last inequality is valid if $q\geq7$.
    By Lemma \ref{Lm_SubspaceIsCodeWord} and Theorem \ref{Res_SzonyiZsuzsa}, $\restr{c}{\pi}$ is a linear combination of exactly $\left\lfloor\frac{\wt(\restr{c}{\pi})}{q+1}\right\rfloor\leq\Delta_{n,q}$ lines of $\pi$.
    If $l$ is one of these lines, then $l$ contains at least $(q+1)-(\Delta_{n,q}-1)$ points of $\supp(c)$, contradicting the assumptions on $m$.
    If $l$ is not one of these lines, then $l$ contains at most $\Delta_{n,q}$ points of $\supp(c)$, yet again a contradiction to the assumptions.
    
    In conclusion, any plane through $l$ has to contain at least $\Delta_{n,q}(q+1)+1$ points of $\supp(c)$.
    As there exist $\theta_{n-2,q}$ planes through $l$, we obtain the following contradiction:
    \begin{align*}
        \wt(c)&\geq\Delta_{n,q}(q+1)+1+(\theta_{n-2,q}-1)\big(\Delta_{n,q}(q+1)+1-m\big)\\
        &\geq\Delta_{n,q}(q+1)+1+(\theta_{n-2,q}-1)(\Delta_{n,q}q-q)\\
        &=\Delta_{n,q}+q+1+q\theta_{n-2,q}(\Delta_{n,q}-1)\\
        &=q+2+\theta_{n-1,q}(\Delta_{n,q}-1)>W(n,q)\textnormal{.}
    \end{align*}
    The second part of the lemma follows directly, as $\Delta_{n,q}\leq W(1,q)$ and $U(n,1,q)=q-\Delta_{n,q}+2$.
\end{proof}

\begin{lm}\label{Lm_ExistsSecant}
    Let $c\in\mc_{n-1}(n,q)$ with $\wt(c)\leq W(n,q)$. Then the value
    \[
        m:=\max\{a\in\NN:\textnormal{there exists a thin }a\textnormal{-secant to }\supp(c)\}
    \]
    is well-defined and belongs to $\{0,1,\dots,\Delta_{n,q}\}$.
    Moreover, $m=0$ if and only if $c=\zero$.
\end{lm}
\begin{proof}
    The value $m$ is clearly well-defined if the set over which the $\max$-operator is taken is a non-empty, finite set.
    If the set is empty, by Lemma \ref{Lm_SmallLargeSecants}, all lines through a certain point $P\notin\supp(c)$ would intersect $\supp(c)$ in at least $q-\Delta_{n,q}+2$ points, hence we would obtain
    \begin{equation}\label{Eq_ThickSecantContradiction}
        \wt(c)\geq\theta_{n-1,q}(q-\Delta_{n,q}+2)>W(n,q)\textnormal{,}
    \end{equation}
    a contradiction.
    By Lemma \ref{Lm_SmallLargeSecants}, $m\in\{0,1,\dots,\Delta_{n,q}\}$.
    Note that if $m=0$, then $c=\zero$, as else we can consider all lines through a certain point $P\in\supp(c)$ and obtain a contradiction similar to \eqref{Eq_ThickSecantContradiction}.
\end{proof}

\begin{lm}\label{Lm_ThinSpaceUpperBound}
    Let $c\in\mc_{n-1}(n,q)$ with $\wt(c)\leq W(n,q)$ and let $\kappa$ be a thin $k$-space, $k\in\{1,2,\dots,n\}$, such that $\restr{c}{\kappa}$ is a linear combination of at most $\left\lceil\frac{\wt(\restr{c}{\kappa})}{\theta_{k-1,q}}\right\rceil$ different $(k-1)$-subspaces of $\kappa$.
    Consider the (by Lemma \ref{Lm_SubspaceIsCodeWord} and \ref{Lm_ExistsSecant} well-defined) value
    \[
        m_\kappa:=\max\{a\in\NN:\textnormal{there exists a thin }a\textnormal{-secant to }\supp(\restr{c}{\kappa})\}\textnormal{.}
    \]
    Then $\restr{c}{\kappa}$ is a linear combination of precisely $m_\kappa$ $(k-1)$-subspaces of $\kappa$.
    As a consequence,
    \[
        \wt(\restr{c}{\kappa})\leq m_\kappa q^{k-1}+\theta_{k-2,q}\leq\Delta_{n,q}q^{k-1}+\theta_{k-2,q}\textnormal{.}
    \]
\end{lm}
\begin{proof}
    Note that the second inequality follows from Lemma \ref{Lm_ExistsSecant}.
    
    If $k=1$, then $\restr{c}{\kappa}$ is obviously a linear combination of $\wt(\restr{c}{\kappa})$ points of the line $\kappa$ and, as $m_\kappa=\wt(\restr{c}{\kappa})$, the first inequality is trivially true.
    Hence, we can assume that $k\geq2$.
    
    By assumption, $\restr{c}{\kappa}$ is a linear combination of at most $\left\lceil\frac{W(k,q)}{\theta_{k-1,q}}\right\rceil=\Delta_{k,q}-1$ different $(k-1)$-subspaces of $\kappa$; denote this set of subspaces by $K$.
    Note that $|K|\geq m_\kappa$, as else there cannot exist an $m_\kappa$-secant to $\supp(\restr{c}{\kappa})$, in contradiction with the definition of $m_\kappa$.
    
    Furthermore, denote the set of points that are contained in at least $i$ subspaces of $K$ by $\mathcal{P}_i(K)$ ($i=1,2$).
    Note that given a $k$-subspace of $\kappa$, other $k$-subspaces each cover at most $q^k$ additional points.
    Inductively exploiting this argument yields that any set of $k$-subspaces cover the largest number of points when sharing a common $(k-2)$-subspace of $\kappa$.
    Using this, together with the fact that $|K|\leq \Delta_{k,q}-1$, we get
    \begin{equation}\label{Eq_PointsInOneSubspaces}
        \wt(\restr{c}{\kappa})\leq|\mathcal{P}_1(K)|\leq|K|q^{k-1}+\theta_{k-2,q}\leq\frac{1}{2^{k-2}}q^{k-1}\sqrt{q}+\theta_{k-2,q}\leq q^{k-1}\sqrt{q}+\theta_{k-2,q}<\theta_{k,q}\textnormal{,}
    \end{equation}
    where the latter inequalities hold as $k\geq2$.
    Moreover, as any two distinct subspaces of the set $K$ have at most $\theta_{k-2,q}$ points in common, we get
    \begin{equation}\label{Eq_PointsInTwoSubspaces}
        |\mathcal{P}_2(K)|\leq\binom{|K|}{2}\theta_{k-2,q}=\frac{|K|(|K|-1)}{2}\theta_{k-2,q}\leq\frac{1}{2}\left(\frac{1}{2^{k-2}}\sqrt{q}\right)^2\theta_{k-2,q}=\frac{1}{2^{2k-3}}q\theta_{k-2,q}<\frac{1}{2}\theta_{k-1,q}\textnormal{,}
    \end{equation}
    where the last inequality holds due to $k\geq2$.
    
    By \eqref{Eq_PointsInOneSubspaces}, we can consider a point $Q\in\kappa\setminus\mathcal{P}_1(K)$.
    If every line in $\kappa$ through $Q$ contains a point of $\mathcal{P}_2(K)$, then $|\mathcal{P}_2(K)|\geq\theta_{k-1,q}$, which contradicts \eqref{Eq_PointsInTwoSubspaces}.
    Hence, there exists a line $l$ in $\kappa$ through $Q$ which intersects $\supp(c)$ only in points of $\mathcal{P}_1(K)\setminus\mathcal{P}_2(K)$.
    As a result, $l$ is a $|K|$-secant to $\supp(c)$.
    By the definition of $m_\kappa$ and Lemma \ref{Lm_SmallLargeSecants}, either $|K|\leq m_\kappa$ or $q-\Delta_{n,q}+2\leq|K|\leq\Delta_{k,q}-1$.
    As $n\geq3$ and $k\geq2$, the latter implies that
    \[
        q-\frac{1}{2}\sqrt{q}+2\leq\sqrt{q}-1\textnormal{,}
    \]
    which is false for any $q$.
    Hence, we conclude that $|K|\leq m_\kappa$ and therefore $|K|= m_\kappa$.
   
    The desired upper bound on $\wt(\restr{c}{\kappa})$ is obtained by the first two inequalities of \eqref{Eq_PointsInOneSubspaces}.
\end{proof}

As a side note, we remind the reader of the following basic result.

\begin{prop}\label{Prop_FloorTrick}
    For any $a,b\in\mathbb{R}$ we have that $\lfloor a\rfloor+\lfloor b\rfloor\leq\lfloor a+b\rfloor$.
    As a consequence, for any $m\in\NN$ and $a\in\mathbb{R}$, we obtain that $m\lfloor a\rfloor\leq\lfloor ma\rfloor$.
    In particular, we get that $2\Delta_{n,q}\leq\Delta_{n-1,q}$.
\end{prop}

The following lemma will be crucial in the proof of our main result.
\begin{lm}\label{Lm_ThinOrThick}
    Let $c\in\mc_{n-1}(n,q)$ with $\wt(c)\leq W(n,q)$ and let $\kappa$ be a $k$-space, $k\in\{1,2,\dots,n-1\}$, with the property that for every thin $j$-subspace $\gamma$ of $\kappa$, $j\in\{1,2,\dots,k-1\}$, $\restr{c}{\gamma}$ is a linear combination of exactly $\left\lceil\frac{\wt(\restr{c}{\gamma})}{\theta_{j-1,q}}\right\rceil$ different $(j-1)$-subspaces of $\gamma$.
    Then $\kappa$ is either thin or thick with respect to $c$.
\end{lm}
\begin{proof}
    We will prove this by induction on $k$.
    The base case $k=1$ follows from Lemma \ref{Lm_SmallLargeSecants}.
    Hence, we can assume that $k\geq2$ and that every $j$-subspace of $\kappa$, $1\leq j<k$, is either thin or thick with respect to $c$.
    
    Suppose that $\kappa$ is not thin; we will prove that $\kappa$ is thick w.r.t.\ $c$.
    
    \bigskip
    \underline{Claim $1$}: $\wt(\restr{c}{\kappa})>\frac{3}{4}q^k-\frac{3}{8}q^{k-1}\sqrt{q}+\frac{3}{4}q^{k-1}$.
    
    \bigskip
    Consider an arbitrary $(k-2)$-subspace $\gamma$ of $\kappa$.
    If every $(k-1)$-subspace of $\kappa$ through $\gamma$ is thick, by Proposition \ref{Prop_ThickSpaceLowerBound}, we obtain
    \[
        \wt(\restr{c}{\kappa})\geq(q+1)(q^{k-1}-\Delta_{n,q}q^{k-2}+1)\textnormal{.}
    \]
    One can easily check that the above inequality proves the claim.
    
    If we denote by $x$ the number of $(k-1)$-subspaces of $\kappa$ through $\gamma$ that are thick, we can now assume that $x<q+1$; hence consider a $(k-1)$-subspace $\lambda$ through $\gamma$ that is not thick, then $\lambda$ is thin by the induction hypothesis.
    By Lemma \ref{Lm_ThinSpaceUpperBound}, this means that $\lambda$ contains at most $\Delta_{n,q}q^{k-2}+\theta_{k-3,q}$ points of $\supp(\restr{c}{\kappa})$.
    However, by the known structure of $\restr{c}{\lambda'}$ for every thin $(k-1)$-subspace $\lambda'$ through $\gamma$ (namely, $\restr{c}{\lambda'}$ is a linear combination of at most $\Delta_{n,q}$ $(k-2)$-subspaces), we know that such $\lambda'$ contains at most $\Delta_{n,q}q^{k-2}$ points lying in $\lambda'\setminus\gamma$.
    
    Consider all $(k-1)$-spaces through $\gamma$:
    \begin{itemize}
        \item $x$ of them are thick and contain at most $q^{k-1}$ points outside of $\gamma$;
        \item $q-x$ of them are thin and not equal to $\lambda$, and contain at most $\Delta_{n,q}q^{k-2}$ points of $\supp(\restr{c}{\kappa})$ outside of $\gamma$;
        \item $1$ of them is $\lambda$ and contains at most $\Delta_{n,q}q^{k-2}+\theta_{k-3,q}$ points of $\supp(\restr{c}{\kappa})$ ($\gamma$ included).
    \end{itemize}
    
    Hence,
    \begin{align*}
        \wt(\restr{c}{\kappa})&\leq xq^{k-1}+(q-x)\Delta_{n,q}q^{k-2}+\Delta_{n,q}q^{k-2}+\theta_{k-3,q}\\
        &=x(q^{k-1}-\Delta_{n,q}q^{k-2})+\Delta_{n,q}q^{k-1}+\Delta_{n,q}q^{k-2}+\theta_{k-3,q}\textnormal{.}
    \end{align*}
    Therefore, as $\kappa$ is assumed not to be thin, we get
    \[
        (\Delta_{k,q}-1)\theta_{k-1,q}+1=W(k,q)+1\leq\wt(\restr{c}{\kappa})\leq x(q^{k-1}-\Delta_{n,q}q^{k-2})+\Delta_{n,q}q^{k-1}+\Delta_{n,q}q^{k-2}+\theta_{k-3,q}\textnormal{.}
    \]
    Note that $q^{k-1}-\Delta_{n,q}q^{k-2}\geq1$ (as $n\geq3$).
    Hence, we can rearrange the above inequalities, taking into account that $k\leq n-1$, to get
    \begin{align*}
        x&\geq\frac{(\Delta_{k,q}-1)\theta_{k-1,q}+1-\big(\Delta_{n,q}q^{k-1}+\Delta_{n,q}q^{k-2}+\theta_{k-3,q}\big)}{q^{k-1}-\Delta_{n,q}q^{k-2}}\\
        &=\frac{(\Delta_{k,q}-\Delta_{n,q}-1)(q^{k-1}+q^{k-2})+(\Delta_{k,q}-2)\theta_{k-3,q}+1}{q^{k-1}-\Delta_{n,q}q^{k-2}}\\
        &\geq\frac{(\Delta_{n-1,q}-\Delta_{n,q}-1)(q^{k-1}+q^{k-2})+(\Delta_{n-1,q}-2)\theta_{k-3,q}+1}{q^{k-1}-\Delta_{n,q}q^{k-2}}\textnormal{.}
    \end{align*}
    As $\Delta_{n-1,q}\geq2\Delta_{n,q}$ by Proposition \ref{Prop_FloorTrick}, we obtain
    \begin{align*}
        x&\geq\frac{(\Delta_{n,q}-1)q^{k-1}+(\Delta_{n,q}-1)q^{k-2}+(2\Delta_{n,q}-2)\theta_{k-3,q}+1}{q^{k-1}-\Delta_{n,q}q^{k-2}}\\
        &=\Delta_{n,q}-1+\frac{\big(\Delta_{n,q}^2-1\big)q^{k-2}+(2\Delta_{n,q}-2)\theta_{k-3,q}+1}{q^{k-1}-\Delta_{n,q}q^{k-2}}\textnormal{.}
    \end{align*}
    
    Using the fact that $\Delta_{n,q}\geq1$, we get
    \begin{align*}
        x&\geq\Delta_{n,q}-1+\frac{1}{q^{k-1}-\Delta_{n,q}q^{k-2}}\\
        \Longrightarrow\quad x&\geq\Delta_{n,q}\textnormal{,}
    \end{align*}
    as $x$ is an integer.
    Hence, we can find a set $K$ of $\Delta_{n,q}$ distinct, thick $(k-1)$-subspaces of $\kappa$ through $\gamma$.
    Let $\mathcal{P}_1(K)$ be the point set of all points contained in at least one of the subspaces of $K$.
    Note that, as $k\leq n-1$, $|\mathcal{P}_1(K)|\leq W(k,q)$.
    Indeed, one can check that
    \begin{align*}
        |\mathcal{P}_1(K)|\leq\Delta_{n,q}q^{k-1}+\theta_{k-2,q}&\leq(\Delta_{k,q}-1)q^{k-1}+\theta_{k-2,q}\\
        &=(\Delta_{k,q}-1)\theta_{k-1,q} +\theta_{k-2,q}(2-\Delta_{k,q})\\
        &\leq (\Delta_{k,q}-1)\theta_{k-1,q}=W(k,q),
    \end{align*}
    where we used the fact that $\Delta_{k,q}\geq\Delta_{n-1,q}\geq\lfloor\frac{1}{2^{n-3}}\sqrt{q}\rfloor\geq2$ as $q\geq 2^{2n-4}$, which also implies $\Delta_{n,q}\leq\frac{1}{2}\Delta_{k,q}\leq\Delta_{k,q}-1$ (Proposition \ref{Prop_FloorTrick}).
    As $W(k,q)<\wt(\restr{c}{\kappa})$, we can find a point $P\in\supp(\restr{c}{\kappa})\setminus\mathcal{P}_1(K)$.
    
    Each of the spaces in $K$ are thick with respect to $c$, hence, by Proposition \ref{Prop_ThickSpaceLowerBound}, each of these subspaces contains less than $\Delta_{n,q}q^{k-2}$ holes of $\restr{c}{\kappa}$.
    Hence, as $n\geq3$, there exist at most $|K|\Delta_{n,q}q^{k-2}+\theta_{k-2,q}\leq\Delta_{n,q}^2q^{k-2}+\theta_{k-2,q}\leq\frac{1}{4}q^{k-1}+\theta_{k-2,q}$ points of $\mathcal{P}_1(K)$ that are either holes of $\restr{c}{\kappa}$ or lie in $\gamma$.
    As a result, there are at least $\theta_{k-1,q}-\frac{1}{4}q^{k-1}-\theta_{k-2,q}=\frac{3}{4}q^{k-1}$ lines in $\kappa$ through $P$ that intersect each subspace of $K$ in a distinct point of $\supp(c)$, since such lines are not contained in any subspace of $K$.
    As each of these lines intersects $\supp(c)$ in at least $\Delta_{n,q}+1$ points, by Lemma \ref{Lm_SmallLargeSecants}, we get
    \begin{align*}
        \wt(\restr{c}{\kappa})&\geq\frac{3}{4}q^{k-1}(q-\Delta_{n,q}+2-1)+1\\
        &=\frac{3}{4}q^{k}-\frac{3}{4}\Delta_{n,q}q^{k-1}+\frac{3}{4}q^{k-1}+1\\
        &>\frac{3}{4}q^{k}-\frac{3}{8}q^{k-1}\sqrt{q}+\frac{3}{4}q^{k-1}\textnormal{,}
    \end{align*}
    proving Claim $1$.
    
    \bigskip
    \underline{Claim $2$}: either $\kappa$ is thick, or there exists a thin $(k-1)$-subspace of $\kappa$.
    
    \bigskip
    Suppose that there does not exist a thin $(k-1)$-subspace of $\kappa$.
    Consider an arbitrary $(k-2)$-subspace $\gamma$ of $\kappa$.
    By the induction hypothesis, every $(k-1)$-subspace of $\kappa$ through $\gamma$ is thick.
    Hence, we obtain
    \begin{equation}\label{Eq_InitialLowerBound}
        \wt(\restr{c}{\kappa})\geq(q+1)U(n,k-1,q)-q\theta_{k-2,q}\textnormal{.}
    \end{equation}
    First, let us consider the case $k=2$.
    If all points of the plane $\kappa$ are points of $\supp(c)$, then $\kappa$ is trivially thick w.r.t.\ $c$ and the proof of the claim is done.
    Hence, we can rechoose $\gamma$ to be a hole of $\restr{c}{\kappa}$ to improve the inequality of \eqref{Eq_InitialLowerBound} and obtain
    \[
        \wt(\restr{c}{\kappa})\geq(q+1)U(n,1,q)=q^2-\Delta_{n,q}q+3q-\Delta_{n,q}+2=U(n,2,q)+q-\Delta_{n,q}+2\geq U(n,2,q)\textnormal{.}
    \]
    Hence, we can assume that $k\geq3$ and expand the right-hand side of \eqref{Eq_InitialLowerBound} to obtain the following.
    \begin{align*}
        \wt(\restr{c}{\kappa})&\geq(q+1)\Big(q^{k-1}-(\Delta_{n,q}-2)q^{k-2}-(k-3)\big((q-1)\Delta_{n,q}+1\big)\lfloor q^{k-4}\rfloor+\theta_{k-4,q}\Big)-q\theta_{k-2,q}\\
        &=q^k-(\Delta_{n,q}-2)q^{k-1}-(k-3)\big((q-1)\Delta_{n,q}+1\big)q\lfloor q^{k-4}\rfloor\\
        &\qquad+q^{k-1}-(\Delta_{n,q}-2)q^{k-2}-(k-3)\big((q-1)\Delta_{n,q}+1\big)\lfloor q^{k-4}\rfloor\\
        &\qquad+q\theta_{k-4,q}+\theta_{k-4,q}-q^{k-1}- q^{k-2}-q\theta_{k-4,q}\textnormal{.}
    \end{align*}
    Since $\lfloor q^{k-4}\rfloor\leq q^{k-4}$,
    \begin{align*}
        \wt(\restr{c}{\kappa})&\geq U(n,k,q)+\big((q-1)\Delta_{n,q}-q\big)q^{k-3}-(\Delta_{n,q}-2)q^{k-2}-(k-3)\big((q-1)\Delta_{n,q}+1\big)q^{k-4}\\
        &= U(n,k,q)+q^{k-2}-(k-2)\Delta_{n,q}q^{k-3}+(k-3)\Delta_{n,q}q^{k-4}-(k-3)q^{k-4}\\
        &\geq U(n,k,q)+q^{k-2}-(n-3)\Delta_{n,q}q^{k-3}\textnormal{,}
    \end{align*}
    where we used the facts that $k\leq n-1$ and $\Delta_{n,q}\geq1$ to get the latter inequality.
    As one can easily check that $\frac{n-3}{2^{n-2}}\leq\frac{1}{4}$, we get
    \[
        \wt(\restr{c}{\kappa})\geq U(n,k,q)+q^{k-2}-\frac{1}{4}q^{k-3}\sqrt{q}= U(n,k,q)+q^{k-3}\sqrt{q}\left(\sqrt{q}-\frac{1}{4}\right)\geq U(n,k,q)\textnormal{.}
    \]
    
    \bigskip
    \underline{Claim $3$}: either $\kappa$ is thick, or there exists a thin $(k-2)$-subspace of $\kappa$ that is contained in at most one thin $(k-1)$-subspace of $\kappa$.
    
    \bigskip
    Suppose to the contrary that $\kappa$ is not thick and suppose that every thin $(k-2)$-subspace of $\kappa$ is contained in at least two distinct, thin $(k-1)$-subspaces of $\kappa$.
    By Claim $2$, there exists a thin $(k-1)$-subspace $\lambda$ of $\kappa$.
    
    We first prove Claim $3$ in case $k=2$.
    As $\lambda$ is thin, by Lemma \ref{Lm_SmallLargeSecants}, $\lambda$ is an $m$-secant to $\supp(c)$ with $m\leq\Delta_{n,q}$; hence, $\lambda$ contains at least $q-\Delta_{n,q}+1$ holes of $\supp(\restr{c}{\kappa})$.
    Moreover, the assumed negation of Claim $3$ implies that through each of these holes there exists at least one other $m'$-secant to $\supp(\restr{c}{\kappa})$ with $m'\leq\Delta_{n,q}$.
    As a consequence, we obtain that
    \begin{align*}
        |\kappa\setminus\supp(\restr{c}{\kappa})|&\geq(q-\Delta_{n,q}+1)+(q-\Delta_{n,q})+(q-\Delta_{n,q}-1)+\dots+1\\
        &=\frac{(q-\Delta_{n,q}+1)(q-\Delta_{n,q}+2)}{2}\\
        &=\frac{1}{2}q^2-\Delta_{n,q}q+\frac{3}{2}q+\frac{1}{2}\Delta_{n,q}^2-\frac{3}{2}\Delta_{n,q}+1\\
        &\geq\frac{1}{2}q^2-\frac{1}{2}q\sqrt{q}+\frac{3}{2}q-\frac{1}{2}\sqrt{q}\textnormal{,}
    \end{align*}
    where we used the facts that $\Delta_{n,q}\leq\frac{1}{2}\sqrt{q}$ and $\Delta_{n,q}\geq1$ to find the latter inequality.
    Combining this with Claim $1$, we get
    \[
        \theta_{2,q}=\wt(\restr{c}{\kappa})+|\kappa\setminus\supp(\restr{c}{\kappa})|>\frac{3}{4}q^2-\frac{3}{8}q\sqrt{q}+\frac{3}{4}q+\frac{1}{2}q^2-\frac{1}{2}q\sqrt{q}+\frac{3}{2}q-\frac{1}{2}\sqrt{q}\textnormal{,}
    \]
    which implies that
    \[
        0>\frac{1}{4}q^2-\frac{7}{8}q\sqrt{q}+\frac{5}{4}q-\frac{1}{2}\sqrt{q}-1\textnormal{,}
    \]
    which is true only if $q<4$, a contradiction.
    
    Now, suppose that $k\geq3$ for the remainder of the proof of Claim $3$.
    Although this case is very similar to the $k=2$ case, a slightly better bound on the number of holes in $\kappa$ was needed to obtain a contradiction for all considered values of $q$.
    First we claim that, within $\lambda$, there exist at least $q+1$ distinct, thin $(k-2)$-subspaces.
    
    Indeed, suppose that there exist at most $q$ thin $(k-2)$-subspaces in $\lambda$.
    As these subspaces cover at most $q\theta_{k-2,q}=\theta_{k-1,q}-1$ points of $\lambda$, we can find a point $P\in\lambda$ that is not contained in a thin $(k-2)$-subspace of $\lambda$.
    When considering a $(k-3)$-subspace $\iota$ of $\lambda$ through $P$ (this is possible as $k\geq3$), we can conclude that this subspace is not contained in any thin $(k-2)$-subspace of $\lambda$.
    By the induction hypothesis, all $(k-2)$-subspaces of $\lambda$ through $\iota$ are thick. Hence, making use of Proposition \ref{Prop_ThickSpaceLowerBound}, we obtain the following contradiction:
    \begin{align*}
        \wt(\restr{c}{\lambda})&\geq(q+1)(\theta_{k-2,q}-\Delta_{n,q}q^{k-3}+1)-q\theta_{k-3,q}\\
        &=\theta_{k-1,q}-\Delta_{n,q}q^{k-2}-\Delta_{n,q}q^{k-3}+q+1\\
        &>W(k-1,q)\textnormal{.}
    \end{align*}
    Hence, we can consider $q+1$ distinct, thin $(k-2)$-subspaces in $\lambda$.
    By the assertion we have made at the beginning of the proof of this claim, each of these subspaces must be contained in at least one thin $(k-1)$-subspace of $\kappa$ other than $\lambda$.
    Each of those $q+1$ thin $(k-1)$-subspaces contain at least $\theta_{k-1,q}-W(k-1,q)\geq q^{k-1}-q^{k-2}\sqrt{q}$ holes (using that $\Delta_{k-1,q}\leq\sqrt{q}$ as $k\geq3$).
    As two distinct $(k-1)$-subspaces intersect in at most a $(k-2)$-space, we can estimate the number of holes in $\kappa$ w.r.t.\ $c$ as follows:
    \begin{align*}
        |\kappa\setminus\supp(\restr{c}{\kappa})|&\geq(q+1)(q^{k-1}-q^{k-2}\sqrt{q})-\binom{q+1}{2}\theta_{k-2,q}\\
        &\geq(q+1)(q^{k-1}-q^{k-2}\sqrt{q})-\frac{q^k+3q^{k-1}}{2}\\
        &\geq\frac{1}{2}q^k-q^{k-1}\sqrt{q}-\frac{1}{2}q^{k-1}-q^{k-2}\sqrt{q}\textnormal{.}
    \end{align*}
    Combining this with Claim $1$, we get
    \[
        \theta_{k,q}=\wt(\restr{c}{\kappa})+|\kappa\setminus\supp(\restr{c}{\kappa})|>\frac{3}{4}q^k-\frac{3}{8}q^{k-1}\sqrt{q}+\frac{3}{4}q^{k-1}+\frac{1}{2}q^k-q^{k-1}\sqrt{q}-\frac{1}{2}q^{k-1}-q^{k-2}\sqrt{q}\textnormal{,}
    \]
    which implies that
    \[
        0>\frac{1}{4}q^k-\frac{7}{8}q^{k-1}\sqrt{q}-\frac{3}{4}q^{k-1}-q^{k-2}\sqrt{q}-\theta_{k-2,q}\textnormal{.}
    \]
    Using the fact that $\theta_{k-2,q}\leq\frac{1}{4}q^{k-2}\sqrt{q}$ (as $q\geq18$), we obtain
    \[
        0>\frac{1}{4}q\sqrt{q}-\frac{7}{8}q-\frac{3}{4}\sqrt{q}-\frac{5}{4}\textnormal{,}
    \]
    which is true only if $q\leq19$, a contradiction.
    
    \bigskip
    \underline{Claim $4$}: $\kappa$ is thick.
    
    \bigskip
    By Claim $3$, we can assume that there exists a thin $(k-2)$-subspace of $\kappa$ that is contained in at most one thin $(k-1)$-subspace of $\kappa$.
    This means that, by the induction hypothesis, this thin $(k-2)$-subspace is contained in at least $q$ thick $(k-1)$-subspaces of $\kappa$.
    Using Lemma \ref{Lm_ThinSpaceUpperBound} and Proposition \ref{Prop_FloorTrick}, we can conclude that
    \begin{align*}
        \wt(\restr{c}{\kappa})&\geq q\cdot U(n,k-1,q)-(q-1)\left(\Delta_{n,q}\lfloor q^{k-3}\rfloor+\theta_{k-4,q}\right)\\
        &=q^k-(\Delta_{n,q}-2)q^{k-1}-(k-3)\big((q-1)\Delta_{n,q}+1\big)q\lfloor q^{k-4}\rfloor+q\theta_{k-4,q}-(q-1)\left(\Delta_{n,q}\lfloor q^{k-3}\rfloor+\theta_{k-4,q}\right)\\
        &\geq q^k-(\Delta_{n,q}-2)q^{k-1}-(k-3)\big((q-1)\Delta_{n,q}+1\big)\lfloor q^{k-3}\rfloor-(q-1)\Delta_{n,q}\lfloor q^{k-3}\rfloor+\theta_{k-4,q}\\
        &=q^k-(\Delta_{n,q}-2)\lfloor q^{k-1}\rfloor-(k-2)\big((q-1)\Delta_{n,q}+1\big)\lfloor q^{k-3}\rfloor+\lfloor q^{k-3}\rfloor+\theta_{k-4,q}\\
        &=U(n,k,q)\textnormal{,}
    \end{align*}
    where we made use that $\lfloor q^{k-3}\rfloor+\theta_{k-4,q}=\theta_{k-3,q}$ for all $k\geq2$ to prove the latter equality.
\end{proof}

We're now fully equipped with the necessary tools to prove the main theorem of this section (Theorem \ref{Thm_MainHyperplane}).

\begin{thm*}
    Any $c\in\mc_{n-1}(n,q)$ with $\wt(c)\leq W(n,q)$ is a linear combination of exactly $\left\lceil\frac{\wt(c)}{\theta_{n-1,q}}\right\rceil$ different hyperplanes.
\end{thm*}
\begin{proof}
    This will be proven by induction on $n$.
    The base case is exactly Lemma \ref{Lm_PlaneCase}.
    Hence, we can assume that $n\geq3$ (as we generally assumed for all previous lemmas of this section) and assume that for every $k$-space $\kappa$, $1\leq k<n$, for which $\wt(\restr{c}{\kappa})\leq W(k,q)$, the codeword $\restr{c}{\kappa}$ is a linear combination of exactly $\left\lceil\frac{\wt(c)}{\theta_{k-1,q}}\right\rceil$ different $(k-1)$-subspaces of $\kappa$.
    As a consequence, by Lemma \ref{Lm_ThinOrThick}, every $k$-space, $1\leq k<n$, is either thin or thick w.r.t.\ $c$.
    
    By Lemma \ref{Lm_ExistsSecant}, we can define the value
    \[
        m:=\max\{a\in\NN:\textnormal{there exists a thin }a\textnormal{-secant to }\supp(c)\}\textnormal{.}
    \]
    If $c=\zero$, then the proof is done.
    Hence, by Lemma \ref{Lm_ExistsSecant}, we can assume that $m\in\{1,2,\dots,\Delta_{n,q}\}$.
    
    \bigskip
    \underline{Claim $1$}: $\wt(c)\geq mq^{n-1}-\frac{m(m-3)}{2}\theta_{n-2,q}$.
    
    \bigskip
    Consider an $m$-secant $l$ and a plane $\sigma$ through $l$.
    This plane is either thin or thick.
    If $\sigma$ is thick, then, by Proposition \ref{Prop_ThickSpaceLowerBound}, $\wt(\restr{c}{\sigma})\geq\theta_{2,q}-\Delta_{n,q}q+2\geq\Delta_{n,q}(q+1)\geq m(q+1)$.
    If $\sigma$ is thin, then we can use the induction hypothesis (as $n\geq3$) and Lemma \ref{Lm_ThinSpaceUpperBound} (as $l\subseteq\sigma$) to see that $\restr{c}{\sigma}$ is a linear combination of precisely $m$ lines of $\sigma$.
    As a consequence,
    \begin{equation}\label{Eq_LowerBoundPlaneSigma}
        \wt(\restr{c}{\sigma})\geq m(q+1)-\binom{m}{2}\textnormal{.}
    \end{equation}
    As $\sigma$ was chosen arbitrarily through $l$, we obtain the desired lower bound on $\wt(c)$:
    \[
        \wt(c)\geq\theta_{n-2,q}\left(m(q+1)-\binom{m}{2}-m\right)+m=mq^{n-1}-\frac{m(m-3)}{2}\theta_{n-2,q}\stepcounter{equation}\textnormal{.}
    \]
    
    \bigskip
    \underline{Claim $2$}: there exists a thick hyperplane.
    
    \bigskip
    Assume the contrary: all hyperplanes are thin.
    Let $\pi$ be a thin hyperplane; by Lemma \ref{Lm_ThinSpaceUpperBound}, $\restr{c}{\pi}$ is a linear combination of at most $m>0$ $(n-2)$-subspaces of $\pi$.
    Let $\gamma$ be one of these $(n-2)$-subspaces.
    Any hyperplane $\pi'$ through $\gamma$ is supposed to be thin, hence, by Lemma \ref{Lm_ThinSpaceUpperBound}, all points of $\supp(\restr{c}{\pi'})$ are covered by $\gamma$ and at most $m-1$ other $(n-2)$-subspaces of $\pi'$.
    In conclusion, we obtain the following upper bound on $\wt(c)$:
    \[
        \wt(c)\leq (q+1)(m-1)q^{n-2}+\theta_{n-2,q}=(m-1)q^{n-1}+mq^{n-2}+\theta_{n-3,q}\textnormal{.}
    \]
    Combining this with Claim $1$, we obtain
    \begin{align*}
        (m-1)q^{n-1}+mq^{n-2}+\theta_{n-3,q}&\geq mq^{n-1}-\frac{m(m-3)}{2}\theta_{n-2,q}\\
        \Longleftrightarrow\quad0&\geq2q^{n-1}-m(m-1)q^{n-2}-(m^2-3m+2)\theta_{n-3,q}.
    \end{align*}
    This yields 
    \begin{align*}
    0&\geq2q^{n-1}-m^2q^{n-2}-2m^2\theta_{n-3,q}
    \end{align*}
    implying that
    \begin{align*}
     0&\geq2q^{n-1}-\frac{1}{4}q^{n-1}-\frac{1}{2}q\theta_{n-3,q},
    \end{align*}
    where we used the fact that $m\leq\Delta_{n,q}\leq\frac{1}{2}\sqrt{q}$ to obtain the last inequality.
    As $q\geq18$, one can check that $\theta_{n-3,q}\leq\frac{1}{4}q^{n-3}\sqrt{q}$.
    Hence, we get
    \begin{align*}
        0&\geq\frac{7}{4}q^{n-1}-\frac{1}{8}q^{n-2}\sqrt{q},
    \end{align*}
    which yields $1\geq14\sqrt{q}$, a clear contradiction.

    \bigskip
    \underline{Claim $3$}: there are at least $(q-\sqrt{q})q^{n-3}$ thin planes through a fixed $m$-secant (w.r.t.\ $\supp(c)$).
    
    \bigskip
    Let $l$ be an $m$-secant to $\supp(c)$ and denote by $x$ the number of thin planes through $l$.
    By \eqref{Eq_LowerBoundPlaneSigma}, we know that each such plane contains at least $m(q+1)-\binom{m}{2}$ points of $\supp(c)$.
    Hence, by Proposition \ref{Prop_ThickSpaceLowerBound}, we obtain
    \begin{align*}
        W(n,q)\geq\wt(c)&\geq x\left(m(q+1)-\binom{m}{2}-m\right)+(\theta_{n-2,q}-x)(\theta_{2,q}-\Delta_{n,q}q+2-m)+m\\
        &=-x\left(q^2-\Delta_{n,q}q-(m-1)q+\frac{m^2-3m+6}{2}\right)+\theta_{n-2,q}\big(q^2-\Delta_{n,q}q+q-(m-3)\big)+m\\
        &\geq-x\left(q^2-\Delta_{n,q}q+q-(m-3)\right)+\theta_{n-2,q}\big(q^2-\Delta_{n,q}q+q-(m-3)\big)\\
        &=\left(\theta_{n-2,q}-x\right)\left(q^2-\Delta_{n,q}q+q-(m-3)\right)\\
        &\geq\left(\theta_{n-2,q}-x\right)\left(q^2-q\sqrt{q}\right)\textnormal{,}
    \end{align*}
    where we used the fact that $m\leq\Delta_{n,q}\leq q$ and that $-m^2+\frac{m^2-3m+6}{2}\leq-(m-3)$ to prove the third inequality and the fact that $\theta_{n-2,q}-x\geq0$ to justify the truth of the last inequality.
    Now suppose, to the contrary, that $x<(q-\sqrt{q})q^{n-3}$.
    Then
    \begin{align*}
        W(n,q)&>\left(\theta_{n-2,q}-(q-\sqrt{q})q^{n-3}\right)\left(q^2-q\sqrt{q}\right)\geq q^{n-1}(\sqrt{q}-1)\geq \theta_{n-1,q}\left(\frac{1}{2}\sqrt{q}-1\right)\geq W(n,q)\textnormal{,}
    \end{align*}
    a contradiction.

    \bigskip
    \underline{Claim $4$}: there are at least $(q-\sqrt{q})q^{n-3}$ thick lines in a thick hyperplane $\Pi$ through a fixed point.
    
    \bigskip
    Denote by $y$ the number of thick lines in $\Pi$ through a point $P\in\Pi$.
    Making use of Lemma \ref{Lm_SmallLargeSecants}, we get
    \[
        (\theta_{n-2,q}-y)\Delta_{n,q}+yq+1\geq\wt(\restr{c}{\Pi})\textnormal{.}
    \]
    As $\Pi$ is thick, by Proposition \ref{Prop_ThickSpaceLowerBound}, we know that $\wt(\restr{c}{\Pi})\geq\theta_{n-1,q}-\Delta_{n,q}q^{n-2}+1$.
    Combining this with the inequality above, we get
    \[
        y(q-\Delta_{n,q})+\Delta_{n,q}\theta_{n-2,q}\geq\theta_{n-1,q}-\Delta_{n,q}q^{n-2}
    \]
    Suppose, to the contrary, that $y<(q-\sqrt{q})q^{n-3}$.
    Then we get
    \begin{equation*}
       (q-\sqrt{q})q^{n-3}(q-\Delta_{n,q})+\Delta_{n,q}\theta_{n-2,q}>\theta_{n-1,q}-\Delta_{n,q}q^{n-2}
      \end{equation*}
    which is equivalent to
    \begin{equation*}
      0>q^{n-2}\left(\sqrt{q}-\Delta_{n,q}-\frac{\Delta_{n,q}}{\sqrt{q}}\right)+(q-\Delta_{n,q})\theta_{n-3,q}+1,
    \end{equation*}
   a contradiction.

    \bigskip
    \underline{Claim $5$}: there exist a thick hyperplane $\Pi$ and an $m$-secant $l\nsubseteq\Pi$ to $\supp(c)$ that intersect in a point contained in $\supp(c)$.
    
    \bigskip
    By Claim $2$, there exists a thick hyperplane $\Pi$; let $l$ be an $m$-secant to $\supp(c)$.
    Note that, for any thin plane $\pi$ through $l$, by the induction hypothesis (as $n\geq3$) and Lemma \ref{Lm_ThinSpaceUpperBound} (as $l\subseteq\pi$), $\restr{c}{\pi}$ is a linear combination of precisely $m$ lines of $\pi$.
    
    First, suppose that $l\subseteq\Pi$.
    Then there are precisely $\theta_{n-3,q}$ planes in $\Pi$ through $l$.
    Hence, by Claim $3$, there exist at least $(q-\sqrt{q})q^{n-3}-\theta_{n-3,q}\geq(q-\sqrt{q})q^{n-3}-\frac{1}{4}q^{n-3}\sqrt{q}=(q-\frac{5}{4}\sqrt{q})q^{n-3}>0$ thin planes through $l$ that are not contained in $\Pi$. Let $\pi$ be one of such planes; $\restr{c}{\pi}$ is a linear combination of precisely $m$ lines of $\pi$.
    Hence, through a point of $l$, there exists another $m$-secant to $\supp(c)$ lying in $\pi$, as $\binom{m}{2}\leq m^2\leq\Delta_{n,q}^2\leq\frac{1}{4}q<q$.
    By replacing $l$ with this newly-found $m$-secant, we can assume that $l\nsubseteq\Pi$.
    Define $P:=l\cap\Pi$ and suppose, to the contrary, that $P\notin\supp(c)$.
    
    Note that, for any thin plane $\pi$ through $l$, the line $\pi\cap\Pi$ cannot be one of the lines present in the linear combination $\restr{c}{\pi}$, as the $m$-secant $l$ has to intersect all such lines in a point of $\supp(c)$.
    As a consequence, for each thick line $t$ in $\Pi$ through $P$, the plane $\vspan{t,l}$ has to be thick.
    
    By Claim $4$, there are at least $(q-\sqrt{q})q^{n-3}$ such thick lines in $\Pi$ through $P$.
    Thus, by Proposition \ref{Prop_ThickSpaceLowerBound}, we obtain
    \begin{align*}
        \wt(c)&\geq(q-\sqrt{q})q^{n-3}(\theta_{2,q}-\Delta_{n,q}q+1-m)+m\\
        &\geq(q-\sqrt{q})q^{n-3}\left(q^2-\frac{1}{2}q\sqrt{q}+q-\frac{1}{2}\sqrt{q}+2\right)\\
        &>W(n,q)\textnormal{,}
    \end{align*}
    where the latter inequality holds for all values of $q$, resulting in a contradiction.
    
    \bigskip
    \underline{Claim $6$}: there exists a hyperplane that contains more than $\frac{1}{2}\theta_{n-1,q}$ points, each of which having the same non-zero value under $c$.
    
    \bigskip
    By Claim $5$, there exist a thick hyperplane $\Pi$ and an $m$-secant $l$ to $\supp(c)$ such that their intersection $P:=l\cap\Pi$ is a point of $\supp(c)$.
    By Claim $3$, there are at least $(q-\sqrt{q})q^{n-3}$ thin planes through $l$.
    By Claim $4$, there are at least $(q-\sqrt{q})q^{n-3}$ thick lines in $\Pi$ through $P$.
    As each plane through $l$ intersects $\Pi$ in a line through $P$, and as there are $\theta_{n-2,q}$ planes through $l$ in total, we can conclude that there must be at least
    \[
        2(q-\sqrt{q})q^{n-3}-\theta_{n-2,q}=q^{n-2}-2q^{n-3}\sqrt{q}-\theta_{n-3,q}\geq q^{n-2}-2q^{n-3}\sqrt{q}-\frac{1}{4}q^{n-3}\sqrt{q}=q^{n-2}-\frac{9}{4}q^{n-3}\sqrt{q}
    \]
    thin planes through $l$ that intersect $\Pi$ in a thick line through $P$. 
    
    As $P\in\supp(c)$, $c(P)=\alpha$ for a certain non-zero value $\alpha\in\mathbb{F}_p^*$.
    Moreover, if $\pi$ is a thin plane through $l$, there must be a unique thick line of $\pi$ going through $P$ (or else $\restr{c}{\pi}$ is a linear combination of $>m$ lines).
    By the assumption hypothesis and Lemma \ref{Lm_SmallLargeSecants}, this unique thick line  contains at least $q-\Delta_{n,q}+2$ points having value $\alpha$ under $c$.
    
    As a result, at least
    \begin{equation*}
       \left(q^{n-2}-\frac{9}{4}q^{n-3}\sqrt{q}\right)(q-\Delta_{n,q}+2-1)+1\geq\left(q^{n-2}-\frac{9}{4}q^{n-3}\sqrt{q}\right)\left(q-\frac{1}{2}\sqrt{q}+1\right)>\frac{1}{2}\theta_{n-1,q}
    \end{equation*}
    points of $\Pi$ have the same non-zero value $\alpha$ under $c$ (the latter inequality holds as $q\geq27$).
    
    \bigskip
    \underline{Claim $7$}: if $c$ is a linear combination of at most $\Delta_{n,q}$ different hyperplanes, then $c$ is a linear combination of precisely $\left\lceil\frac{\wt(c)}{\theta_{n-1,q}}\right\rceil$ different hyperplanes.
    
    \bigskip
    Let $c$ be a linear combination of precisely $j\leq\Delta_{n,q}$ different hyperplanes.
    If $j=0$, the claim is trivially true, as $\wt(c)=0$ in this case.
    Hence, we can assume that $j\geq1$.
    We will derive a lower and upper bound on the weight of $c$.
    
    Firstly, as every two different hyperplanes have $\theta_{n-2,q}$ points in common, we can naively state that
    \[
        \wt(c)\geq j\theta_{n-1,q}-\binom{j}{2}\theta_{n-2,q}\textnormal{.}
    \]
    As $\binom{j}{2}\leq\frac{1}{2}j^2$ and $j\leq\Delta_{n,q}\leq\frac{1}{2}\sqrt{q}$, we can deduce from the above inequality that
    \begin{equation}\label{Eq_WeightLowerBound}
        \wt(c)\geq j\theta_{n-1,q}-\frac{1}{8}q\theta_{n-2,q}=j\theta_{n-1,q}-\frac{1}{8}\theta_{n-1,q}+\frac{1}{8}\geq \left(j-\frac{1}{8}\right)\theta_{n-1,q}.
    \end{equation}
    Secondly, as each hyperplane contains $\theta_{n-1,q}$ points, we obtain
    \begin{equation}\label{Eq_WeightUpperBound}
        \wt(c)\leq j\theta_{n-1,q}.
    \end{equation}
    Hence, combining \eqref{Eq_WeightLowerBound} and \eqref{Eq_WeightUpperBound}, we get
    \[
        j=\left\lceil j-\frac{1}{8}\right\rceil\leq\left\lceil\frac{\wt(c)}{\theta_{n-1,q}}\right\rceil\leq j\textnormal{.}
    \]
    
    \bigskip
    \underline{Claim $8$}: Any $c\in\mc_{n-1}(n,q)$ with $\wt(c)\leq W(n,q)$ is a linear combination of at most $\Delta_{n,q}-1$ different hyperplanes.
    
    \bigskip
    Suppose the contrary, and let $c$ be a codeword of minimal weight with the property that $\wt(c)\leq W(n,q)$ and that $c$ cannot be written as a linear combination of at most $\Delta_{n,q}-1$ different hyperplanes.
    By Claim $6$, there exists a hyperplane $\Pi$ that contains more than $\frac{1}{2}\theta_{n-1,q}$ points, all having the same non-zero value $\alpha\in\mathbb{F}_p^*$ under $c$.
    As $f_\Pi$ is a codeword of $\mc_{n-1}(n,q)$, $c-\alpha\cdot f_\Pi$ is a codeword as well, with weight strictly smaller than $\wt(c)$.
    By the minimality of $c$, the codeword $c-\alpha\cdot f_\Pi$ has to be a linear combination of at most $\Delta_{n,q}-1$ different hyperplanes, hence $c$ has to be a linear combination of at most $\Delta_{n,q}-1+1=\Delta_{n,q}$ different hyperplanes.
    By Claim $7$, $c$ is a linear combination of precisely $\left\lceil\frac{\wt(c)}{\theta_{n-1,q}}\right\rceil\leq\left\lceil\frac{W(n,q)}{\theta_{n-1,q}}\right\rceil=\Delta_{n,q}-1$ different hyperplanes, contradicting the assumed properties of $c$.
    
    As a result, Claim $8$ is true.
    Claim $7$ and $8$ prove the theorem.
\end{proof}

\section{Minimal codewords}\label{Sec:minimal}

In this section, we will partially characterise the minimal codewords of $\mc_{n-1}(n,q)$.
To this end, we make use of the same assumptions made at the beginning of Section \ref{Sect_MinWeight}, although this time we can include the case $n=2$ (and make use of Theorem \ref{Res_SzonyiZsuzsa}).
As noted below the statement of Theorem \ref{Thm_MainHyperplane}, the assumptions on $q$ need not be explicitly written down for the theorem to be true; the same holds for the theorems arising in this section.
Hence, we may (silently) assume that $n\geq2$, $q=p^h$ with $h\geq2$ and that

\[
    q\geq\begin{cases}\max\left\{32,2^{2n-4}\right\}\quad&\textnormal{if }h>2\textnormal{,}\\2^{2n}\quad&\textnormal{if }h=2\textnormal{.}\end{cases}
\]

Furthermore, we focus on codewords of $\mc_{n-1}(n,q)$ of weight at most $W(n,q)$ (see Section \ref{Sect_MinWeight}).
In light of this, define
\[
    t_c:=\left\lceil\frac{\wt(c)}{\theta_{n-1,q}}\right\rceil
\]
for each codeword $c\in\mc_{n-1}(n,q)$.

\begin{prop}\label{Prop_SubsetHyperplanes}
    Let $c$ and $c'$ be codewords of $\mc_{n-1}(n,q)$, with $\supp(c')\subseteq\supp(c)$ and $\wt(c)\leq W(n,q)$, and suppose that $c$, respectively $c'$, can be written as a linear combination of hyperplanes belonging to a set $\widetilde{\mathcal{H}}_c$, respectively $\widetilde{\mathcal{H}}_{c'}$, with $|\widetilde{\mathcal{H}}_c|=t_c$ and $|\widetilde{\mathcal{H}}_{c'}|=t_{c'}$.
    Then $\widetilde{\mathcal{H}}_{c'}\subseteq\widetilde{\mathcal{H}}_c$.
    
    As a consequence the following holds.
    \begin{enumerate}
        \item There exists a unique set of $t_c$ hyperplanes such that $c$ can be written as a linear combination of these hyperplanes; we will denote this set of hyperplanes by ${\mathcal{H}_c}$.
        Any codeword $c$ with $\wt(c)\leq W(n,q)$ uniquely determines such a set $\mathcal{H}_c$.
        \item Let $\mathcal{H}_c=\{H_1,\dots,H_{t_c}\}$ be the unique set of hyperplanes such that $c$ can be written as a linear combination of hyperplanes of $\mathcal{H}_c$.
        Then the coefficients in the corresponding linear combination are uniquely determined by $c$.
        If $c=\sum_{i=1}^{t_c}\alpha_iH_i$, then we will write 
        \begin{equation}\label{Eq_ValueOfHyperplane}
            c(H):=\begin{cases}\alpha_i\quad\textnormal{if }H=H_i\textnormal{,}\\0\quad\textnormal{otherwise,}\end{cases}
        \end{equation}
        for any hyperplane $H$ of  $\pg(n,q)$.
    \end{enumerate}
\end{prop}
\begin{proof}
    Note that by Theorem \ref{Res_SzonyiZsuzsa} and Theorem \ref{Thm_MainHyperplane}, both $c$ and $c'$ can be written as a linear combination of exactly $t_c$, respectively $t_{c'}$, different hyperplanes (in other words, the coefficients corresponding to each of these hyperplanes w.r.t.\ $c$, respectively $c'$, are non-zero).
    
    Suppose, to the contrary, that there exists a hyperplane $H\in\widetilde{\mathcal{H}}_{c'}\setminus\widetilde{\mathcal{H}}_c$.
    Then all hyperplanes in $\left(\widetilde{\mathcal{H}}_c\cup\widetilde{\mathcal{H}}_{c'}\right)\setminus\{H\}$ cover at most
    \[
        \left(2t_c-1\right)\theta_{n-2,q}\leq\left(2\left\lceil\frac{W(n,q)}{\theta_{n-1,q}}\right\rceil-1\right)\theta_{n-2,q}\leq(2\Delta_{n,q}-3)\theta_{n-2,q}\leq(2\sqrt{q}-3)\theta_{n-2,q}<\theta_{n-1,q}
    \]
    points of $H$.
    As a consequence, there exists a point $P\in H$ which is not contained in any hyperplane of $\left(\widetilde{\mathcal{H}}_c\cup\widetilde{\mathcal{H}}_{c'}\right)\setminus\{H\}$.
    As $P$ is not contained in any hyperplane of $\widetilde{\mathcal{H}}_c$, $P\notin\supp(c)$.
    However, as $H\in\widetilde{\mathcal{H}}_{c'}$ and as $P$ is contained in no other hyperplane of $\widetilde{\mathcal{H}}_{c'}$ except for $H$, $P\in\supp(c')\subseteq\supp(c)$, a contradiction.
    
    Statement $1.$ follows immediately by defining $c':=c$.
    Statement $2.$ follows by repeating the above arguments for the unique set $\mathcal{H}_c$: we observe that each hyperplane of $\mathcal{H}_c$ contains a point that is not contained in any other hyperplane of $\mathcal{H}_c$, hence its coefficient w.r.t.\ $c$ is uniquely determined.
\end{proof}

\begin{df}\label{Def_RestrictedHyperplaneCodeword}
    Let $c$ be a codeword of $\mc_{n-1}(n,q)$ with $\wt(c)\leq W(n,q)$ and consider its set of hyperplanes $\mathcal{H}_c$.
    Suppose $\mathcal{H}\subseteq\mathcal{H}_c$.
    Keeping the extended definition \eqref{Eq_ValueOfHyperplane} of $c$ in mind, we can define\footnote{This shouldn't interfere with the definition of a \emph{restricted codeword} $\restr{c}{\iota}$, where $\iota$ is an $i$-subspace.}
    \[
        \restr{c}{\mathcal{H}}:=\sum_{H\in\mathcal{H}}c(H)H\textnormal{.}
     \]
    In particular, we can state the trivial expression $\restr{c}{\mathcal{H}_c}=c$.
\end{df}

\begin{df}\label{Def_Minimal}
    Let $c\in\mc_{n-1}(n,q)$.
    Then $c$ is \emph{minimal} if for each $c'\in\mc_{n-1}(n,q)$ with $\supp(c')\subseteq\supp(c)$, there exists an $\alpha\in\mathbb{F}_p$ such that $c'=\alpha c$.
\end{df}

\begin{df}\label{Def_GraphAdjacency}
    Let $c\in\mc_{n-1}(n,q)$ with $\wt(c)\leq W(n,q)$ and suppose $\mathbb{H}_c$ is a partition of $\mathcal{H}_c$.
    Consider the graph $\Gamma_{\mathbb{H}_c}$ with vertex set $\mathbb{H}_c$, where two vertices $\mathcal{V}_1$ and $\mathcal{V}_2$ are adjacent if and only if there exists a point $P\in\pg(n,q)$ such that
    \begin{enumerate}
        \item $P$ is a hole of $c$,
        \item $P$ belongs to the support of both $\restr{c}{\mathcal{V}_1}$ and $\restr{c}{\mathcal{V}_2}$, and
        \item $P$ is a hole of $\restr{c}{\mathcal{V}}$ for any $\mathcal{V}\in\mathbb{H}_c\setminus\{\mathcal{V}_1,\mathcal{V}_2\}$.
    \end{enumerate}
\end{df}

\begin{constr}\label{Constr_HyperplaneGraphPartition}
    Let $c\in\mc_{n-1}(n,q)$ with $\wt(c)\leq W(n,q)$ and suppose $\mathbb{H}_c^0:=\binom{\mathcal{H}_c}{1}$ is the set of singletons, each containing a unique hyperplane in $\mathcal{H}_c$.
    For each $i\in\NN$, we recursively define $\mathbb{H}_c^{i+1}$ in the following way:
    \[
        \mathbb{H}_c^{i+1}:=\left\{\bigcup_{\mathcal{V}\in\mathbb{V}}\mathcal{V}:\mathbb{V}\textnormal{ is the vertex set of a connected component in }\Gamma_{\mathbb{H}_c^i}\right\}\textnormal{.}
    \]
    Note that this set is yet again a partition of $\mathcal{H}_c$, and that $\mathbb{H}_c^i$ is a refinement of $\mathbb{H}_c^{i+1}$ for all $i\in\NN$.
    Hence, as $\mathcal{H}_c$ is finite, there exists a $j\in\NN$ such that $\mathbb{H}_c^j=\mathbb{H}_c^{j+1}=\mathbb{H}_c^{j+2}=\dots$; denote the latter set by $\mathbb{H}_c^\infty$.
\end{constr}

\begin{figure}
    \begin{center}\begin{tikzpicture}
        \draw[thin, line join=round, line cap=round, name path=a0] (270+40:1) -- (90+40:4.5);
        \node[draw=none, fill=none, anchor=east] at (90+40:4.5) {$a_0$};
        \node[draw=none, fill=none, anchor=west] at (90+40:4.5) {$\boldsymbol{\alpha}$};
        
        \draw[thin, line join=round, line cap=round] (270+15:0.8) -- (90+15:3.5);
        \node[draw=none, fill=none, anchor=east] at (90+15:3.5) {$\widetilde{a}$};
        \node[draw=none, fill=none, anchor=west] at (90+15:3.5) {$\boldsymbol{2\alpha}$};
        
        \draw[thin, line join=round, line cap=round, name path=a1] (90+15:2.5) -- +(180-8:2);
        \draw[thin, line join=round, line cap=round] (90+15:2.5) -- +(-8:3);
        \draw[thin, line join=round, line cap=round, name path=a2] (90+15:2.5) -- +(180+8:1.5);
        \draw[thin, line join=round, line cap=round] (90+15:2.5) -- +(8:3.5);
        \draw[fill=white] (90+15:2.5) circle (2.5pt);
        \node[draw=none, fill=none, anchor=east] at ($(90+15:2.5)+(180-8:2)$) {$a_1$};
        \node[draw=none, fill=none, anchor=east] at ($(90+15:2.5)+(180+8:1.5)$) {$a_2$};
        \node[draw=none, fill=none, anchor=west] at ($(90+15:2.5)+(-8:3)$) {$\boldsymbol{-\alpha}$};
        \node[draw=none, fill=none, anchor=west] at ($(90+15:2.5)+(8:3.5)$) {$\boldsymbol{-\alpha}$};
        
        \draw [fill=white, name intersections={of=a0 and a1}] (intersection-1) +(-2pt,-2pt) rectangle +(2pt,2pt);
        \draw [fill=white, name intersections={of=a0 and a2}] (intersection-1) +(-2pt,-2pt) rectangle +(2pt,2pt);
        
        \draw[thin, line join=round, line cap=round, name path=b0] (270-40:1) -- (90-40:4.5);
        \node[draw=none, fill=none, anchor=east] at (90-40:4.5) {$b_0$};
        \node[draw=none, fill=none, anchor=west] at (90-40:4.5) {$\boldsymbol{\beta}$};
        
        \draw[thin, line join=round, line cap=round] (270-15:0.8) -- (90-15:3.5);
        \node[draw=none, fill=none, anchor=east] at (90-15:3.5) {$\widetilde{b}$};
        \node[draw=none, fill=none, anchor=west] at (90-15:3.5) {$\boldsymbol{3\beta}$};
        
        \draw[thin, line join=round, line cap=round] (90-15:1) -- +(180-20:2);
        \draw[thin, line join=round, line cap=round, name path=b1] (90-15:1) -- +(-20:1);
        \draw[thin, line join=round, line cap=round] (90-15:1) -- +(180:2);
        \draw[thin, line join=round, line cap=round, name path=b2] (90-15:1) -- +(0:1.5);
        \draw[thin, line join=round, line cap=round] (90-15:1) -- +(180+20:1.5);
        \draw[thin, line join=round, line cap=round, name path=b3] (90-15:1) -- +(20:1.5);
        \draw[fill=white] (90-15:1) circle (2.5pt);
        \node[draw=none, fill=none, anchor=east] at ($(90-15:1)+(180-20:2)$) {$b_1$};
        \node[draw=none, fill=none, anchor=east] at ($(90-15:1)+(180:2)$) {$b_2$};
        \node[draw=none, fill=none, anchor=east] at ($(90-15:1)+(180+20:1.5)$) {$b_3$};
        \node[draw=none, fill=none, anchor=west] at ($(90-15:1)+(-20:1)$) {$\boldsymbol{-\beta}$};
        \node[draw=none, fill=none, anchor=west] at ($(90-15:1)+(0:1.5)$) {$\boldsymbol{-\beta}$};
        \node[draw=none, fill=none, anchor=west] at ($(90-15:1)+(20:1.5)$) {$\boldsymbol{-\beta}$};
        
        \draw [fill=white, name intersections={of=b0 and b1}] (intersection-1) +(-2pt,-2pt) rectangle +(2pt,2pt);
        \draw [fill=white, name intersections={of=b0 and b2}] (intersection-1) +(-2pt,-2pt) rectangle +(2pt,2pt);
        \draw [fill=white, name intersections={of=b0 and b3}] (intersection-1) +(-2pt,-2pt) rectangle +(2pt,2pt);
        
        \draw[fill=white] (0,0) circle (3pt);
    \end{tikzpicture}
    \quad
    \begin{tikzpicture}
        \draw[very thick, line join=round, line cap=round, name path=a0] (270+40:1) -- (90+40:4.5);
        \node[draw=none, fill=none, anchor=east] at (90+40:4.5) {$a_0$};
        \node[draw=none, fill=none, anchor=west] at (90+40:4.5) {$\boldsymbol{\alpha}$};
        
        \draw[thin, line join=round, line cap=round] (270+15:0.8) -- (90+15:3.5);
        \node[draw=none, fill=none, anchor=east] at (90+15:3.5) {$\widetilde{a}$};
        \node[draw=none, fill=none, anchor=west] at (90+15:3.5) {$\boldsymbol{2\alpha}$};
        
        \draw[very thick, line join=round, line cap=round, name path=a1] (90+15:2.5) -- +(180-8:2);
        \draw[very thick, line join=round, line cap=round] (90+15:2.5) -- +(-8:3);
        \draw[very thick, line join=round, line cap=round, name path=a2] (90+15:2.5) -- +(180+8:1.5);
        \draw[very thick, line join=round, line cap=round] (90+15:2.5) -- +(8:3.5);
        \draw[fill=white] (90+15:2.5) +(-2.5pt,-2.5pt) rectangle +(2.5pt,2.5pt);
        \node[draw=none, fill=none, anchor=south west] at (90+15:2.5) {$A$};
        \node[draw=none, fill=none, anchor=east] at ($(90+15:2.5)+(180-8:2)$) {$a_1$};
        \node[draw=none, fill=none, anchor=east] at ($(90+15:2.5)+(180+8:1.5)$) {$a_2$};
        \node[draw=none, fill=none, anchor=west] at ($(90+15:2.5)+(-8:3)$) {$\boldsymbol{-\alpha}$};
        \node[draw=none, fill=none, anchor=west] at ($(90+15:2.5)+(8:3.5)$) {$\boldsymbol{-\alpha}$};
        
        \draw [very thick, fill=white, name intersections={of=a0 and a1}] (intersection-1) circle (2pt);
        \draw [very thick, fill=white, name intersections={of=a0 and a2}] (intersection-1) circle (2pt);
        
        \draw[very thick, dashed, line join=round, line cap=round, name path=b0] (270-40:1) -- (90-40:4.5);
        \node[draw=none, fill=none, anchor=east] at (90-40:4.5) {$b_0$};
        \node[draw=none, fill=none, anchor=west] at (90-40:4.5) {$\boldsymbol{\beta}$};
        
        \draw[thin, line join=round, line cap=round] (270-15:0.8) -- (90-15:3.5);
        \node[draw=none, fill=none, anchor=east] at (90-15:3.5) {$\widetilde{b}$};
        \node[draw=none, fill=none, anchor=west] at (90-15:3.5) {$\boldsymbol{3\beta}$};
        
        \draw[very thick, dashed, line join=round, line cap=round] (90-15:1) -- +(180-20:2);
        \draw[very thick, dashed, line join=round, line cap=round, name path=b1] (90-15:1) -- +(-20:1);
        \draw[very thick, dashed, line join=round, line cap=round] (90-15:1) -- +(180:2);
        \draw[very thick, dashed, line join=round, line cap=round, name path=b2] (90-15:1) -- +(0:1.5);
        \draw[very thick, dashed, line join=round, line cap=round] (90-15:1) -- +(180+20:1.5);
        \draw[very thick, dashed, line join=round, line cap=round, name path=b3] (90-15:1) -- +(20:1.5);
        \draw[fill=white] (90-15:1) +(-2.5pt,-2.5pt) rectangle +(2.5pt,2.5pt);
        \node[draw=none, fill=none, anchor=east] at ($(90-15:1)+(180-20:2)$) {$b_1$};
        \node[draw=none, fill=none, anchor=east] at ($(90-15:1)+(180:2)$) {$b_2$};
        \node[draw=none, fill=none, anchor=east] at ($(90-15:1)+(180+20:1.5)$) {$b_3$};
        \node[draw=none, fill=none, anchor=west] at ($(90-15:1)+(-20:1)$) {$\boldsymbol{-\beta}$};
        \node[draw=none, fill=none, anchor=west] at ($(90-15:1)+(0:1.5)$) {$\boldsymbol{-\beta}$};
        \node[draw=none, fill=none, anchor=west] at ($(90-15:1)+(20:1.5)$) {$\boldsymbol{-\beta}$};
        
        \draw [very thick, fill=white, name intersections={of=b0 and b1}] (intersection-1) circle (2pt);
        \draw [very thick, fill=white, name intersections={of=b0 and b2}] (intersection-1) circle (2pt);
        \draw [very thick, fill=white, name intersections={of=b0 and b3}] (intersection-1) circle (2pt);
        
        \draw[fill=white] (0,0) circle (3pt);
    \end{tikzpicture}
    \begin{tikzpicture}
        \clip (-8,-0.5) rectangle (8,0.5);
        \node[draw,very thick,fill=none,anchor=west] at (-1.5,0) {$\boldsymbol{3\alpha+4\beta=0}$};
    \end{tikzpicture}
    \begin{tikzpicture}
        \draw[very thick, line join=round, line cap=round, name path=a0] (270+40:1) -- (90+40:4.5);
        \node[draw=none, fill=none, anchor=east] at (90+40:4.5) {$a_0$};
        \node[draw=none, fill=none, anchor=west] at (90+40:4.5) {$\boldsymbol{\alpha}$};
        
        \draw[very thick, line join=round, line cap=round] (270+15:0.8) -- (90+15:3.5);
        \node[draw=none, fill=none, anchor=east] at (90+15:3.5) {$\widetilde{a}$};
        \node[draw=none, fill=none, anchor=west] at (90+15:3.5) {$\boldsymbol{2\alpha}$};
        
        \draw[very thick, line join=round, line cap=round, name path=a1] (90+15:2.5) -- +(180-8:2);
        \draw[very thick, line join=round, line cap=round] (90+15:2.5) -- +(-8:3);
        \draw[very thick, line join=round, line cap=round, name path=a2] (90+15:2.5) -- +(180+8:1.5);
        \draw[very thick, line join=round, line cap=round] (90+15:2.5) -- +(8:3.5);
        \draw[very thick, fill=white] (90+15:2.5) circle (2.5pt);
        \node[draw=none, fill=none, anchor=east] at ($(90+15:2.5)+(180-8:2)$) {$a_1$};
        \node[draw=none, fill=none, anchor=east] at ($(90+15:2.5)+(180+8:1.5)$) {$a_2$};
        \node[draw=none, fill=none, anchor=west] at ($(90+15:2.5)+(-8:3)$) {$\boldsymbol{-\alpha}$};
        \node[draw=none, fill=none, anchor=west] at ($(90+15:2.5)+(8:3.5)$) {$\boldsymbol{-\alpha}$};
        
        \draw [very thick, fill=white, name intersections={of=a0 and a1}] (intersection-1) circle (2pt);
        \draw [very thick, fill=white, name intersections={of=a0 and a2}] (intersection-1) circle (2pt);
        
        \draw[very thick, dashed, line join=round, line cap=round, name path=b0] (270-40:1) -- (90-40:4.5);
        \node[draw=none, fill=none, anchor=east] at (90-40:4.5) {$b_0$};
        \node[draw=none, fill=none, anchor=west] at (90-40:4.5) {$\boldsymbol{\beta}$};
        
        \draw[very thick, dashed, line join=round, line cap=round] (270-15:0.8) -- (90-15:3.5);
        \node[draw=none, fill=none, anchor=east] at (90-15:3.5) {$\widetilde{b}$};
        \node[draw=none, fill=none, anchor=west] at (90-15:3.5) {$\boldsymbol{3\beta}$};
        
        \draw[very thick, dashed, line join=round, line cap=round] (90-15:1) -- +(180-20:2);
        \draw[very thick, dashed, line join=round, line cap=round, name path=b1] (90-15:1) -- +(-20:1);
        \draw[very thick, dashed, line join=round, line cap=round] (90-15:1) -- +(180:2);
        \draw[very thick, dashed, line join=round, line cap=round, name path=b2] (90-15:1) -- +(0:1.5);
        \draw[very thick, dashed, line join=round, line cap=round] (90-15:1) -- +(180+20:1.5);
        \draw[very thick, dashed, line join=round, line cap=round, name path=b3] (90-15:1) -- +(20:1.5);
        \draw[very thick, fill=white] (90-15:1) circle (2.5pt);
        \node[draw=none, fill=none, anchor=east] at ($(90-15:1)+(180-20:2)$) {$b_1$};
        \node[draw=none, fill=none, anchor=east] at ($(90-15:1)+(180:2)$) {$b_2$};
        \node[draw=none, fill=none, anchor=east] at ($(90-15:1)+(180+20:1.5)$) {$b_3$};
        \node[draw=none, fill=none, anchor=west] at ($(90-15:1)+(-20:1)$) {$\boldsymbol{-\beta}$};
        \node[draw=none, fill=none, anchor=west] at ($(90-15:1)+(0:1.5)$) {$\boldsymbol{-\beta}$};
        \node[draw=none, fill=none, anchor=west] at ($(90-15:1)+(20:1.5)$) {$\boldsymbol{-\beta}$};
        
        \draw [very thick, fill=white, name intersections={of=b0 and b1}] (intersection-1) circle (2pt);
        \draw [very thick, fill=white, name intersections={of=b0 and b2}] (intersection-1) circle (2pt);
        \draw [very thick, fill=white, name intersections={of=b0 and b3}] (intersection-1) circle (2pt);
        
        \draw[fill=white] (0,0) +(-3pt,-3pt) rectangle +(3pt,3pt);
    \end{tikzpicture}
    \quad
    \begin{tikzpicture}
        \draw[very thick, line join=round, line cap=round, name path=a0] (270+40:1) -- (90+40:4.5);
        \node[draw=none, fill=none, anchor=east] at (90+40:4.5) {$a_0$};
        \node[draw=none, fill=none, anchor=west] at (90+40:4.5) {$\boldsymbol{\alpha}$};
        
        \draw[very thick, line join=round, line cap=round] (270+15:0.8) -- (90+15:3.5);
        \node[draw=none, fill=none, anchor=east] at (90+15:3.5) {$\widetilde{a}$};
        \node[draw=none, fill=none, anchor=west] at (90+15:3.5) {$\boldsymbol{2\alpha}$};
        
        \draw[very thick, line join=round, line cap=round, name path=a1] (90+15:2.5) -- +(180-8:2);
        \draw[very thick, line join=round, line cap=round] (90+15:2.5) -- +(-8:3);
        \draw[very thick, line join=round, line cap=round, name path=a2] (90+15:2.5) -- +(180+8:1.5);
        \draw[very thick, line join=round, line cap=round] (90+15:2.5) -- +(8:3.5);
        \draw[very thick, fill=white] (90+15:2.5) circle (2.5pt);
        \node[draw=none, fill=none, anchor=east] at ($(90+15:2.5)+(180-8:2)$) {$a_1$};
        \node[draw=none, fill=none, anchor=east] at ($(90+15:2.5)+(180+8:1.5)$) {$a_2$};
        \node[draw=none, fill=none, anchor=west] at ($(90+15:2.5)+(-8:3)$) {$\boldsymbol{-\alpha}$};
        \node[draw=none, fill=none, anchor=west] at ($(90+15:2.5)+(8:3.5)$) {$\boldsymbol{-\alpha}$};
        
        \draw [very thick, fill=white, name intersections={of=a0 and a1}] (intersection-1) circle (2pt);
        \draw [very thick, fill=white, name intersections={of=a0 and a2}] (intersection-1) circle (2pt);
        
        \draw[very thick, line join=round, line cap=round, name path=b0] (270-40:1) -- (90-40:4.5);
        \node[draw=none, fill=none, anchor=east] at (90-40:4.5) {$b_0$};
        \node[draw=none, fill=none, anchor=west] at (90-40:4.5) {$\boldsymbol{\beta}$};
        
        \draw[very thick, line join=round, line cap=round] (270-15:0.8) -- (90-15:3.5);
        \node[draw=none, fill=none, anchor=east] at (90-15:3.5) {$\widetilde{b}$};
        \node[draw=none, fill=none, anchor=west] at (90-15:3.5) {$\boldsymbol{3\beta}$};
        
        \draw[very thick, line join=round, line cap=round] (90-15:1) -- +(180-20:2);
        \draw[very thick, line join=round, line cap=round, name path=b1] (90-15:1) -- +(-20:1);
        \draw[very thick, line join=round, line cap=round] (90-15:1) -- +(180:2);
        \draw[very thick, line join=round, line cap=round, name path=b2] (90-15:1) -- +(0:1.5);
        \draw[very thick, line join=round, line cap=round] (90-15:1) -- +(180+20:1.5);
        \draw[very thick, line join=round, line cap=round, name path=b3] (90-15:1) -- +(20:1.5);
        \draw[very thick, fill=white] (90-15:1) circle (2.5pt);
        \node[draw=none, fill=none, anchor=east] at ($(90-15:1)+(180-20:2)$) {$b_1$};
        \node[draw=none, fill=none, anchor=east] at ($(90-15:1)+(180:2)$) {$b_2$};
        \node[draw=none, fill=none, anchor=east] at ($(90-15:1)+(180+20:1.5)$) {$b_3$};
        \node[draw=none, fill=none, anchor=west] at ($(90-15:1)+(-20:1)$) {$\boldsymbol{-\beta}$};
        \node[draw=none, fill=none, anchor=west] at ($(90-15:1)+(0:1.5)$) {$\boldsymbol{-\beta}$};
        \node[draw=none, fill=none, anchor=west] at ($(90-15:1)+(20:1.5)$) {$\boldsymbol{-\beta}$};
        
        \draw [very thick, fill=white, name intersections={of=b0 and b1}] (intersection-1) circle (2pt);
        \draw [very thick, fill=white, name intersections={of=b0 and b2}] (intersection-1) circle (2pt);
        \draw [very thick, fill=white, name intersections={of=b0 and b3}] (intersection-1) circle (2pt);
        
        \draw[very thick, fill=white] (0,0) circle (3pt);
    \end{tikzpicture}\end{center}
    \caption[caption]{The application of Construction \ref{Constr_HyperplaneGraphPartition} to an example codeword $c\in\mc_1(2,q)$ of weight $9q-12$.
    More specifically, we consider nine lines of $\pg(2,q)$ and define the codeword $c:=\alpha(a_0-a_1-a_2)+2\alpha\,\widetilde{a}+\beta(b_0-b_1-b_2-b_3)+3\beta\,\widetilde{b}$.
    For this specific example, we assume $q$ not prime, $q\geq529$ if $h=2$ (to be able to apply Theorem \ref{Res_SzonyiZsuzsa}) and $q\geq125$, $p\notin\{2,3,7,11,13\}$, if $h>2$.
    Furthermore, $\alpha,\beta\in\mathbb{F}_q^*$ are two non-zero elements such that $3\alpha+4\beta=0$.\\\hspace{\textwidth}
    Lines are clustered in four `stages', each of which consists of `clustering' the lines by following the rule of thumb described in Construction \ref{Constr_HyperplaneGraphPartition}.
    Holes that are about to `merge' clusters are indicated by squares instead of circles.
    In the first stage (top left), every line forms its own cluster.
    In the second stage (top right), the solid bold lines form one cluster, as well as the dashed bold lines; the remaining lines $\widetilde{a}$ and $\widetilde{b}$  form two clusters on their own.
    In the third stage (bottom left), the line $\widetilde{a}$ gets merged into the solid bold cluster and the line $\widetilde{b}$ gets merged into the dashed bold cluster.
    Finally, in the last stage (bottom right), both clusters get merged into one.\\\hspace{\textwidth}
    To explain more clearly how this merging process works, consider the point $A$ in the second stage.
    At this stage,  $A$ belongs to both the support of the solid bold line cluster $\{a_0,a_1,a_2\}$ (with non-zero value $-2\alpha$) and the support of the cluster $\{\widetilde{a}\}$ (with non-zero value $2\alpha$), and thus meets Property $2.$ of Definition \ref{Def_GraphAdjacency}.
    Moreover,  $A$ is a hole of $c$, as well as a hole of $\restr{c}{\mathcal{V}}$ for every other cluster (and therefore fulfils Property $1.$ and $3.$ of Definition \ref{Def_GraphAdjacency}).
    Hence, the clusters $\{a_0,a_1,a_2\}$ and $\{\widetilde{a}\}$ are \emph{adjacent} and thus will get merged in the next stage as Construction \ref{Constr_HyperplaneGraphPartition} prescribes.}
    \label{Fig_Example}
\end{figure}
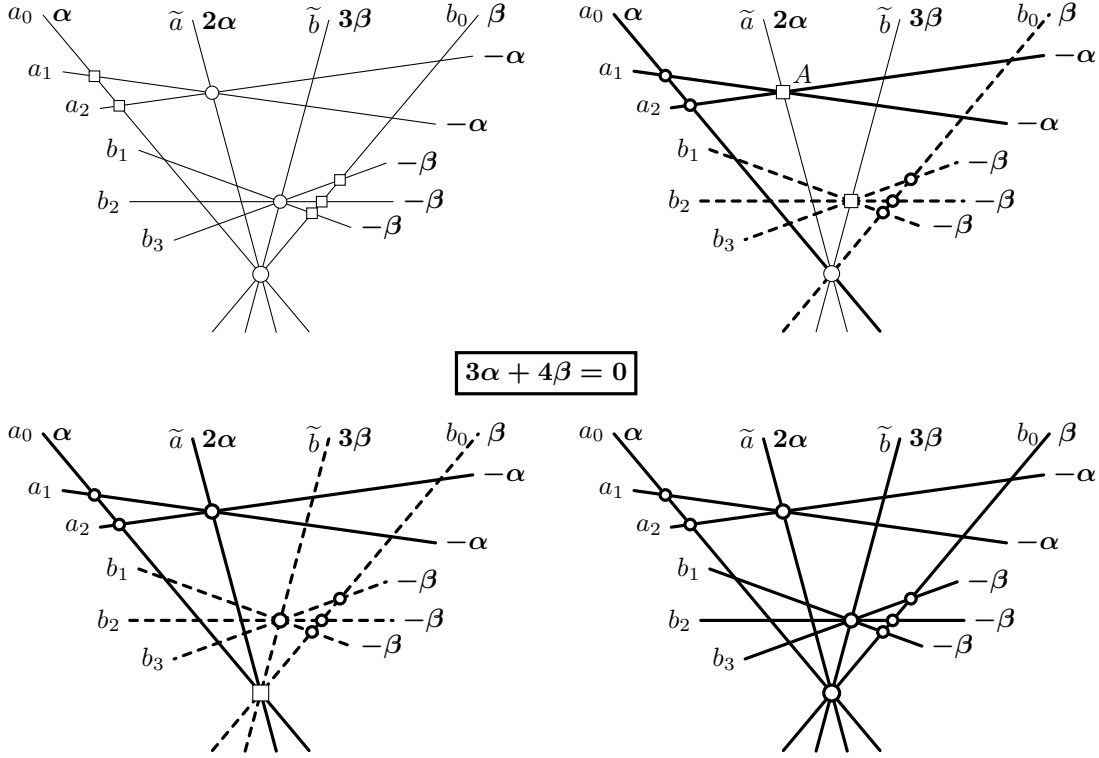

Figure \ref{Fig_Example} is an illustration of the way Construction \ref{Constr_HyperplaneGraphPartition} deals with a specific small weight codeword $c\in\mc_1(2,q)$, $q$ large enough.
The drawing consists of four `stages', and one can check that
\begin{enumerate}
    \item $\mathbb{H}_c^0=\{\{a_0\},\{a_1\},\{a_2\},\{\widetilde{a}\},\{b_0\},\{b_1\},\{b_2\},\{b_3\},\{\widetilde{b}\}\}$,
    \item $\mathbb{H}_c^1=\{\{a_0,a_1,a_2\},\{\widetilde{a}\},\{b_0,b_1,b_2,b_3\},\{\widetilde{b}\}\}$,
    \item $\mathbb{H}_c^2=\{\{a_0,a_1,a_2,\widetilde{a}\},\{b_0,b_1,b_2,b_3,\widetilde{b}\}\}$, and
    \item $\mathbb{H}_c^3=\{\{a_0,a_1,a_2,\widetilde{a},b_0,b_1,b_2,b_3,\widetilde{b}\}\}=\mathbb{H}_c^\infty$.
\end{enumerate}

Hence, for this specific codeword $c$, we end up with $|\mathbb{H}_c^\infty|=1$, a property which, by the following theorem, implies that $c$ is a minimal codeword of $\mc_1(2,q)$.

\begin{thm}
    Let $c\in\mc_{n-1}(n,q)$ with $\wt(c)\leq W(n,q)$ and suppose $|\mathbb{H}_c^{\infty}|=1$.
    Then $c$ is minimal.
\end{thm}
\begin{proof}
    Consider an arbitrary codeword $c'\in\mc_{n-1}(n,q)$ for which $\supp(c')\subseteq\supp(c)$.
    We want to prove that there exists an $\alpha\in\mathbb{F}_p$ such that $c'=\alpha c$.
    Keeping Definition \ref{Def_RestrictedHyperplaneCodeword} in mind, this will be done by proving that
    \begin{equation}\label{Eq_ToProveMinimal}
        \forall i\in\NN, \quad \forall\mathcal{V}\in\mathbb{H}_c^i, \quad  \exists\alpha_i^\mathcal{V}\in\mathbb{F}_p \ : \ \restr{c'}{\mathcal{V}}=\alpha_i^\mathcal{V}\restr{c}{\mathcal{V}}.
    \end{equation}
    Indeed, if \eqref{Eq_ToProveMinimal} is true, then it is true for $i=\infty$ (read: for $i$ large enough).
    As $\mathbb{H}_c^\infty$ only contains the element $\mathcal{H}_c$, and as $\mathcal{H}_{c'}\subseteq\mathcal{H}_c$ (Proposition \ref{Prop_SubsetHyperplanes}), this implies that there exists an $\alpha:=\alpha_\infty^{\mathcal{H}_c}\in\mathbb{F}_p$ such that
    \[
        c'=\restr{c'}{\mathcal{H}_{c'}}=\restr{c'}{\mathcal{H}_c}=\alpha\restr{c}{\mathcal{H}_c}=\alpha c\textnormal{.}
    \]
    We will prove \eqref{Eq_ToProveMinimal} by induction on $i$.
    In case $i=0$, the set $\mathbb{H}_c^0=\binom{\mathcal{H}_c}{1}$ partitions $\mathcal{H}_c$ in singletons.
    Hence, for an arbitrary element $\mathcal{V}\in\mathbb{H}_c^0$, there exists a hyperplane $H\in\mathcal{H}_c$ such that $\mathcal{V}=\{H\}$.
    This means that $\restr{c'}{\mathcal{V}}=c'(H)H$ and $\restr{c}{\mathcal{V}}=c(H)H$.
    As $H\in\mathcal{H}_c$ implies that $c(H)\neq0$, we find an $\alpha_0^\mathcal{V}:=c'(H)c(H)^{-1}$ meeting the requirements.
    
    Now assume that \eqref{Eq_ToProveMinimal} is true for $i\in\NN$.
    Choose an arbitrary $\mathcal{V}\in\mathbb{H}_c^{i+1}$.
    As $\mathbb{H}_c^i$ is a refinement of $\mathbb{H}_c^{i+1}$, $\mathcal{V}=\mathcal{V}'_1\sqcup\dots\sqcup\mathcal{V}'_k$ for certain pairwise disjoint sets of hyperplanes $\mathcal{V}'_1,\dots,\mathcal{V}'_k\in\mathbb{H}_c^i$ ($k\in\{1,\dots,t_c\}$).
    Moreover, the elements of $\{\mathcal{V}'_1,\dots,\mathcal{V}'_k\}$ are precisely all the vertices of a connected component in the graph $\Gamma_{\mathbb{H}_c^i}$.
    Consider two elements from this set that are adjacent w.r.t.\ the graph $\Gamma_{\mathbb{H}_c^i}$; w.l.o.g.\ let these two elements be $\mathcal{V}'_1$ and $\mathcal{V}'_2$.
    By Definition \ref{Def_GraphAdjacency}, there exists a point $P$ such that
    \begin{enumerate}
        \item $c(P)=0$, which implies that $c'(P)=0$ as $\supp(c')\subseteq\supp(c)$,
        \item $\restr{c}{\mathcal{V}'_1}(P)\neq0\neq\restr{c}{\mathcal{V}'_2}(P)$, and
        \item for all $\mathcal{V}'\in\mathbb{H}_c^i\setminus\{\mathcal{V}'_1,\mathcal{V}'_2\}$, $\restr{c}{\mathcal{V}'}(P)=0$, implying that $\restr{c'}{\mathcal{V}'}(P)=0$ by the induction hypothesis.
    \end{enumerate}
    As $\mathbb{H}_c^i$ is a partition of $\mathcal{H}_c$, we know that $c=\restr{c}{\mathcal{H}_c}=\sum_{\mathcal{V}'\in\mathbb{H}_c^i}\restr{c}{\mathcal{V}'}$.
    Moreover, as $\mathcal{H}_{c'}\subseteq\mathcal{H}_c$ (Proposition \ref{Prop_SubsetHyperplanes}), we have $c'=\restr{c'}{\mathcal{H}_{c'}}=\restr{c'}{\mathcal{H}_c}=\sum_{\mathcal{V}'\in\mathbb{H}_c^i}\restr{c'}{\mathcal{V}'}$.
    Hence, by Property $1.$ and $3.$ above, we obtain
    \begin{align}
        0=c(P)=&\left(\sum_{\mathcal{V}'\in\mathbb{H}_c^i}\restr{c}{\mathcal{V}'}\right)(P)=\restr{c}{\mathcal{V}'_1}(P)+\restr{c}{\mathcal{V}'_2}(P)\textnormal{, and}\label{Eq_SumZero1}\\
        0=c'(P)=&\left(\sum_{\mathcal{V}'\in\mathbb{H}_c^i}\restr{c'}{\mathcal{V}'}\right)(P)=\restr{c'}{\mathcal{V}'_1}(P)+\restr{c'}{\mathcal{V}'_2}(P)\textnormal{.}\label{Eq_SumZero2}
    \end{align}
    By the induction hypothesis, there exist elements $\alpha_i^{\mathcal{V}'_1},\alpha_i^{\mathcal{V}'_2}\in\mathbb{F}_p$ such that $\restr{c'}{\mathcal{V}'_1}=\alpha_i^{\mathcal{V}'_1}\restr{c}{\mathcal{V}'_1}$ and $\restr{c'}{\mathcal{V}'_2}=\alpha_i^{\mathcal{V}'_2}\restr{c}{\mathcal{V}'_2}$.
    Combining this with \eqref{Eq_SumZero1}, \eqref{Eq_SumZero2} and the fact that $\restr{c}{\mathcal{V}'_1}(P)\neq0\neq\restr{c}{\mathcal{V}'_2}(P)$ (Property $2.$ above), we obtain that
    \[
        \alpha_i^{\mathcal{V}'_1}=\left(\restr{c}{\mathcal{V}'_1}(P)\right)^{-1}\restr{c'}{\mathcal{V}'_1}(P)=\left(-\restr{c}{\mathcal{V}'_2}(P)\right)^{-1}\left(-\restr{c'}{\mathcal{V}'_2}(P)\right)=\left(\restr{c}{\mathcal{V}'_2}(P)\right)^{-1}\restr{c'}{\mathcal{V}'_2}(P)=\alpha_i^{\mathcal{V}'_2}\textnormal{.}
    \]
    In conclusion, for any two elements $\mathcal{V}'_{j_1},\mathcal{V}'_{j_2}\in\{\mathcal{V}'_1,\dots,\mathcal{V}'_k\}$ that are adjacent w.r.t.\ the graph $\Gamma_{\mathbb{H}_c^i}$, their values $\alpha_i^{\mathcal{V}'_{j_1}}$ and $\alpha_i^{\mathcal{V}'_{j_2}}$ (found by the induction hypothesis) are equal.
    As the elements of $\{\mathcal{V}'_1,\dots,\mathcal{V}'_k\}$ are precisely the vertices of a connected component of $\Gamma_{\mathbb{H}_c^i}$, we can conclude that
    \[
        \restr{c'}{\mathcal{V}'_j}=\alpha_i^{\mathcal{V}'_1}\restr{c}{\mathcal{V}'_j}\textnormal{, }\forall j\in\{1,\dots,k\}\textnormal{.}
    \]
    As a direct consequence, by defining $\alpha_{i+1}^\mathcal{V}:=\alpha_i^{\mathcal{V}'_1}$, we conclude that, for any $\mathcal{V}\in\mathbb{H}_c^{i+1}$,
    \[
        \restr{c'}{\mathcal{V}}=\sum_{j=1}^k\restr{c'}{\mathcal{V}'_j}=\sum_{j=1}^k\alpha_{i+1}^\mathcal{V}\restr{c}{\mathcal{V}'_j}=\alpha_{i+1}^\mathcal{V}\sum_{j=1}^k\restr{c}{\mathcal{V}'_j}=\alpha_{i+1}^\mathcal{V}\restr{c}{\mathcal{V}}\textnormal{.}\qedhere
    \]
\end{proof}

\begin{df}
    Let $c\in\mc_{n-1}(n,q)$ with $\wt(c)\leq W(n,q)$.
    Then we define $\mathcal{P}_c^\infty$ as the set of all holes $P$ (of $c$) for which there exists an $\mathcal{V}\in\mathbb{H}_c^\infty$ such that $\restr{c}{\mathcal{V}}(P)\neq0$.
\end{df}

\begin{thm}\label{Thm_NotMinimal}
    Let $c\in\mc_{n-1}(n,q)$ with $\wt(c)\leq W(n,q)$ and suppose that $|\mathcal{P}_c^\infty|\leq |\mathbb{H}_c^\infty|-2$.
    Then $c$ is not minimal.
\end{thm}
\begin{proof}
    Define $h:=|\mathbb{H}_c^\infty|$ and $r:=|\mathcal{P}_c^\infty|$, and let $\mathbb{H}_c^\infty=\{\mathcal{V}_1,\dots,\mathcal{V}_h\}$ and $\mathcal{P}_c^\infty=\{P_1,\dots,P_r\}$.
    Consider the following system of $r$ linear equations over $\mathbb{F}_p$:
    \begin{equation}\label{Eq_SystemLinearEqs}
        \restr{c}{\mathcal{V}_1}(P_i)X_1+\restr{c}{\mathcal{V}_2}(P_i)X_2+\dots+\restr{c}{\mathcal{V}_h}(P_i)X_h=0\textnormal{,}\qquad i=1,2,\dots,r\textnormal{.}
    \end{equation}
    As $r\leq h-2$, the solution space of the above system of linear equations is a vector space over $\mathbb{F}_p$ of dimension at least two.
    Therefore, we can find a solution $(\alpha_1,\dots,\alpha_h)\in\mathbb{F}_p^h$ that is not a scalar multiple of $(1,\dots,1)\in\mathbb{F}_p^h$.
    Define
    \[
        c':=\alpha_1\restr{c}{\mathcal{V}_1}+\alpha_2\restr{c}{\mathcal{V}_2}+\dots+\alpha_h\restr{c}{\mathcal{V}_h}\textnormal{.}
    \]
    By the choice of $(\alpha_1,\dots,\alpha_h)$, $c'$ is not a scalar multiple of $c=\restr{c}{\mathcal{V}_1}+\restr{c}{\mathcal{V}_2}+\dots+\restr{c}{\mathcal{V}_h}$.
    Hence, once we verify that $\supp(c')\subseteq\supp(c)$, the proof is done.
    
    Consider a hole $Q$ of $c$.
    Then either $Q\in\mathcal{P}_c^\infty$ or $Q\notin\mathcal{P}_c^\infty$.
    If $Q\in\mathcal{P}_c^\infty$, then $Q=P_{i'}$ for an $i'\in\{1,2,\dots,r\}$.
    As $(\alpha_1,\dots,\alpha_h)$ is a solution to System \eqref{Eq_SystemLinearEqs},  $c'(Q)=\alpha_1\restr{c}{\mathcal{V}_1}(P_{i'})+\dots+\alpha_h\restr{c}{\mathcal{V}_h}(P_{i'})=0$.
    If $Q\notin\mathcal{P}_c^\infty$, then $\restr{c}{\mathcal{V}}(Q)=0$ for any $\mathcal{V}\in\mathbb{H}_c^\infty$, hence $c'(Q)=\alpha_1\restr{c}{\mathcal{V}_1}(Q)+\dots+\alpha_h\restr{c}{\mathcal{V}_h}(Q)=0+\dots+0=0$.
    In conclusion, $c'(Q)=0$.
    As $Q$ was an arbitrary hole of $c$, we obtain that the complement of $\supp(c)$ is contained in the complement of $\supp(c')$, hence $\supp(c')\subseteq\supp(c)$.
\end{proof}

\begin{crl}
    Let $c\in\mc_{n-1}(n,q)$ with $\wt(c)\leq W(n,q)$ and suppose that $|\mathbb{H}_c^\infty|=2$.
    Then $c$ is not minimal.
\end{crl}
\begin{proof}
    By Theorem \ref{Thm_NotMinimal}, it suffices to prove that $\mathcal{P}_c^\infty=\emptyset$.
    Let  $\mathbb{H}_c^\infty=\{\mathcal{V}_1,\mathcal{V}_2\}$ and suppose, to the contrary, that there exists a point $P\in\mathcal{P}_c^\infty$.
    W.l.o.g.\ we can assume that $\restr{c}{\mathcal{V}_1}(P)\neq0$.
    As $P\in\mathcal{P}_c^\infty$, we know that $P$ is a hole of $c$, hence $0=c(P)=\restr{c}{\mathcal{V}_1}(P)+\restr{c}{\mathcal{V}_2}(P)$.
    As $\restr{c}{\mathcal{V}_1}(P)\neq0$, $\restr{c}{\mathcal{V}_2}(P)\neq0$ as well.
    However, by Construction \ref{Constr_HyperplaneGraphPartition}, this would imply that $|\mathbb{H}_c^\infty|=1$, a contradiction.
\end{proof}

It seems reasonable to conjecture that, if $p>3$, the bound on the size of $\mathcal{P}_c^\infty$ depicted in Theorem \ref{Thm_NotMinimal} is sharp.

\begin{thm}\label{Th:example}
    Suppose that $p>3$.
    Then there exists a minimal codeword $c\in\mc_1(2,q)$ with $\wt(c)\leq W(2,q)$, $|\mathbb{H}_c^\infty|=3$ and $|\mathcal{P}_c^\infty|=2$.
\end{thm}
\begin{proof}
    We describe the codeword illustrated in Figure \ref{Fig_CounterExample}.
    Consider a line $t$ and let $R,S\in t$ be two distinct points.
    Let $r_1$, $r_2$ and $r'$ be three distinct lines through $R$, different from $t$, and let $s_1$, $s_2$ and $s'$ be three distinct lines through $S$, different from $t$.
    Define $c:=r_1+r_2-r'+s_1+s_2-s'-t$.
    
    One can check that $\mathbb{H}_c^\infty=\left\{\{r_1,r_2,s'\},\{s_1,s_2,r'\},\{t\}\right\}$ and that $\mathcal{P}_c^\infty=\{R,S\}$.
    One can also manually check that this codeword is minimal.
\end{proof}

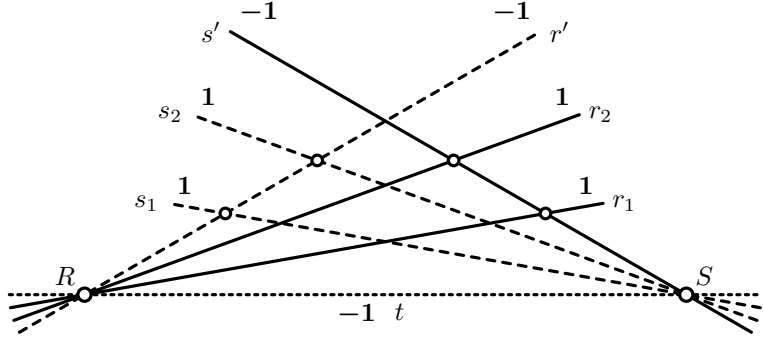
\begin{figure}
    \begin{center}\begin{tikzpicture}
        \draw[very thick, dotted, line join=round, line cap=round] (-5,0) -- (5,0);
        \node[draw=none, fill=none, anchor=north west] at (0,0) {$t$};
        \node[draw=none, fill=none, anchor=north east] at (0,0) {$\boldsymbol{-1}$};
        
        \draw[very thick, line join=round, line cap=round, name path=r1] ($(-4,0)+(180+10:1)$) -- ($(-4,0)+(10:7)$);
        \draw[very thick, line join=round, line cap=round, name path=r2] ($(-4,0)+(180+20:1)$) -- ($(-4,0)+(20:7)$);
        \draw[very thick, dashed, line join=round, line cap=round, name path=rr] ($(-4,0)+(180+30:1)$) -- ($(-4,0)+(30:7)$);
        \draw[very thick, fill=white] (-4,0) circle (2.5pt);
        \node[draw=none, fill=none, anchor=south east] at (-4,0) {$R$};
        \node[draw=none, fill=none, anchor=west] at ($(-4,0)+(10:7)$) {$r_1$};
        \node[draw=none, fill=none, anchor=west] at ($(-4,0)+(20:7)$) {$r_2$};
        \node[draw=none, fill=none, anchor=west] at ($(-4,0)+(30:7)$) {$r'$};
        \node[draw=none, fill=none, anchor=south east] at ($(-4,0)+(10:7)$) {$\boldsymbol{1}$};
        \node[draw=none, fill=none, anchor=south east] at ($(-4,0)+(20:7)$) {$\boldsymbol{1}$};
        \node[draw=none, fill=none, anchor=south east] at ($(-4,0)+(30:7)$) {$\boldsymbol{-1}$};
        
        \draw[very thick, dashed, line join=round, line cap=round, name path=s1] ($(4,0)+(360-10:1)$) -- ($(4,0)+(180-10:7)$);
        \draw[very thick, dashed, line join=round, line cap=round, name path=s2] ($(4,0)+(360-20:1)$) -- ($(4,0)+(180-20:7)$);
        \draw[very thick, line join=round, line cap=round, name path=ss] ($(4,0)+(360-30:1)$) -- ($(4,0)+(180-30:7)$);
        \draw[very thick, fill=white] (4,0) circle (2.5pt);
        \node[draw=none, fill=none, anchor=south west] at (4,0) {$S$};
        \node[draw=none, fill=none, anchor=east] at ($(4,0)+(180-10:7)$) {$s_1$};
        \node[draw=none, fill=none, anchor=east] at ($(4,0)+(180-20:7)$) {$s_2$};
        \node[draw=none, fill=none, anchor=east] at ($(4,0)+(180-30:7)$) {$s'$};
        \node[draw=none, fill=none, anchor=south west] at ($(4,0)+(180-10:7)$) {$\boldsymbol{1}$};
        \node[draw=none, fill=none, anchor=south west] at ($(4,0)+(180-20:7)$) {$\boldsymbol{1}$};
        \node[draw=none, fill=none, anchor=south west] at ($(4,0)+(180-30:7)$) {$\boldsymbol{-1}$};
        
        \draw [very thick, fill=white, name intersections={of=r1 and ss}] (intersection-1) circle (2pt);
        \draw [very thick, fill=white, name intersections={of=r2 and ss}] (intersection-1) circle (2pt);
        \draw [very thick, fill=white, name intersections={of=s1 and rr}] (intersection-1) circle (2pt);
        \draw [very thick, fill=white, name intersections={of=s2 and rr}] (intersection-1) circle (2pt);
    \end{tikzpicture}\end{center}
    \caption{An example of a codeword $c\in\mc_1(2,q)$, $p>3$, that proves the sharpness of the bound of Theorem \ref{Thm_NotMinimal} for the case $n=2$, $|\mathbb{H}_c^\infty|=3$ (see Theorem \ref{Th:example}).}
    \label{Fig_CounterExample}
\end{figure}

\section{Open Problems}\label{Sec:Open}
We conclude this work with a few open problems.

\begin{open}
It would be interesting to generalise the example of Theorem \ref{Th:example} (Figure \ref{Fig_CounterExample}) to higher dimensions $n\geq3$.
Moreover, one could raise the question whether there exist examples of minimal codewords $c\in\mc_1(2,q)$, $\wt(c)\leq W(2,q)$, for which $|\mathcal{P}_c^\infty|=|\mathbb{H}_c^\infty|-1$, with $|\mathbb{H}_c^\infty|\geq4$ (with a generalisation to $n\geq3$ as well).
\end{open}

\begin{open}
If $p=2$, one can observe that the following codewords of $\mc_1(2,q)$ are not minimal.
\begin{enumerate}
    \item At least three lines through a common point.
    \item At least three lines, one of which contains no holes.
\end{enumerate}
In all other cases, the codewords seem to be minimal.
It would be interesting if one could prove that this is true for codewords of $\mc_1(2,2^h)$ up to a relatively large weight.
\end{open}

\textbf{Acknowledgements.}
The research of Daniele Bartoli was supported by the Italian National Group for Algebraic and Geometric Structures and their Applications (GNSAGA - INdAM).

\bibliographystyle{plain}
\bibliography{main}

\end{document}